\documentclass[reqno,11pt]{amsart}
\usepackage{amsmath,amsthm,amssymb,amsfonts}
\usepackage{epsfig,color,graphicx,enumerate,bbm}
\usepackage{accents}
\usepackage[a4paper,margin=2.5truecm]{geometry}

\DeclareMathOperator{\per}{Per}
\DeclareMathOperator{\dist}{dist}

\def\ds{\displaystyle}
\def\eps{{\varepsilon}}
\def\N{\mathbb{N}}
\def\R{\mathbb{R}}

\def\O{\Omega}

\def\Dr{D}
\def\F{\mathcal{F}}
\def\G{\mathcal{G}}
\def\HH{\mathcal{H}}
\def\LL{\mathcal{L}}

\def\PP{\mathcal{P}}

\def\vf{\varphi}

\newcommand{\be}{\begin{equation}}
\newcommand{\ee}{\end{equation}}
\newcommand{\bib}[4]{\bibitem{#1}{\sc#2: }{\it#3. }{#4.}}
\newcommand{\cp}{\mathop{\rm cap}\nolimits}

\newcommand{\ind}{\mathbbm{1}}
\newcommand{\mean}[1]{\,-\hskip-1.08em\int_{#1}} %media integrale displayed
\numberwithin{equation}{section}
\theoremstyle{plain}
\newtheorem{teo}{Theorem}[section]
\newtheorem{lemma}[teo]{Lemma}
\newtheorem{cor}[teo]{Corollary}
\newtheorem{prop}[teo]{Proposition}
\newtheorem{deff}[teo]{Definition}
\theoremstyle{remark}
\newtheorem{oss}[teo]{Remark}

\newtheorem{exam}[teo]{Example}

\newenvironment{ack}{{\bf Acknowledgements.}}

\title[The spectral drop problem]{The spectral drop problem}

\author{Giuseppe Buttazzo}
\address{Dipartimento di Matematica, Universit\`a di Pisa, Largo B. Pontecorvo 5, 56126 Pisa, ITALY}
\email{buttazzo@dm.unipi.it}

\author{Bozhidar Velichkov}
\address{Dipartimento di Matematica, Universit\`a di Pisa, Largo B. Pontecorvo 5, 56126 Pisa, ITALY}
\email{b.velichkov@sns.it}

\begin{document}
\maketitle

\begin{abstract}
We consider spectral optimization problems of the form
$$\min\Big\{\lambda_1(\O;D):\ \O\subset D,\ |\O|=1\Big\},$$
where $D$ is a given subset of the Euclidean space $\R^d$. Here $\lambda_1(\O;D)$ is the first eigenvalue of the Laplace operator $-\Delta$ with Dirichlet conditions on $\partial\O\cap D$ and Neumann or Robin conditions on $\partial\O\cap\partial D$. The equivalent variational formulation
$$\lambda_1(\O;D)=\min\left\{\int_\O|\nabla u|^2\,dx+k\int_{\partial D}u^2\,d\HH^{d-1}\ :\ u\in H^1(D),\ u=0\hbox{ on }\partial\O\cap D,\ \|u\|_{L^2(\O)}=1\right\}$$
reminds the classical drop problems, where the first eigenvalue replaces the total variation functional. We prove an existence result for general shape cost functionals and we show some qualitative properties of the optimal domains.
\end{abstract}

\medskip
\textbf{Keywords:} Shape optimization, spectral cost, drop problems, Dirichlet energy

\textbf{2010 Mathematics Subject Classification:} 49A50, 49G05, 49R50, 49Q10

%%%%%%%%%%%%%%%%%%%%%%%%%%%%%%%%%%%%%%%%%%%%%%%%%%
\section{Introduction}\label{sintro}

We fix an open set $D\subset\R^d$ with a Lipschitz boundary, not necessarily bounded, and a function $f\in L^2(D)$; for every domain $\O\subset D$ we define the Sobolev space
$$H^1_0(\O;D)=\big\{u\in H^1(D)\ :\ u=0\ \hbox{ q.e. on }D\setminus\O\big\}$$
where q.e. means, as usual, up to a set of capacity zero. When $\Dr=\R^d$ we use the notation $H^1_0(\O):=H^1_0(\O;\R^d)$. We also fix a real number $k$ and we define the energy $E_{k,f}(\O)$ by the variational problem
\be\label{energy}
E_{k,f}(\O)=\inf\left\{\frac12\int_D|\nabla u|^2\,dx+\frac k2\int_{\partial D}u^2\,d\HH^{d-1}-\int_Dfu\,dx\ :\ u\in H^1_0(\O;D)\right\}.
\ee
Note that if $\O$ is an open set with $\overline\O\subset D$, then the condition $u\in H^1_0(\O;D)$ is equivalent to require $u\in H^1_0(\O)$. On the contrary, if $\partial\O\cap\partial D\ne\emptyset$ and if the infimum in \eqref{energy} is attained, passing to the Euler-Lagrange equation associated to \eqref{energy} we obtain
\be\label{pde}
\begin{cases}
\begin{array}{rll}
\ds-\Delta u=f&\hbox{in}&\O,\\
\ds u=0&\hbox{on}&\partial\O\cap \Dr,\\
\ds\frac{\partial u}{\partial n}+ku=0&\hbox{on}&\partial\O\cap\partial \Dr.
\end{array}
\end{cases}
\ee
It is not difficult to see that the infimum in \eqref{energy} is attained whenever $k>-k_0(\O)$, where
$$k_0(\O)=\inf\left\{\int_D|\nabla u|^2\,dx\ :\ u\in H^1_0(\O;D),\ \|u\|_{L^2(\partial D)}=1\right\}.$$
Our goal is to study the shape optimization problem
\be\label{shoptene}
\min\Big\{E_{k,f}(\O)\ :\ \O\subset D,\ |\O|\le1\Big\},
\ee
where we have normalized to $1$ the measure constraint on the competing domains $\O$. Of course, the set $D$ is assumed to have a measure larger than $1$.

In the rest of the paper we consider a number $k$ which is not too negative; more precisely, we assume that $k>-k_0$ where
$$k_0=-\inf\big\{k_0(\O)\ :\ \O\subset D,\ |\O|\le1\big\}.$$
If the condition above is violated and $k<-k_0(\O)$ for some $\O$, then it is easy to see that $E_{k,f}(\O)=-\infty$, hence the shape optimization problem \eqref{shoptene} is not well posed. The limit case $k=-k_0$ is more delicate and the well posedness of \eqref{shoptene} depends on the geometry of $D$. A detailed analysis for the shape functional $\lambda_1(\O)$ with Robin boundary conditions can be found in \cite{dan13}.

Replacing $E_{k,f}(\O)$ by another shape functional $\F(\O)$ we may consider the more general class of problems
\be\label{shoptgen}
\min\Big\{\F(\O)\ :\ \O\subset D,\ |\O|\le1\Big\}.
\ee

For an overview on shape optimization problems we refer to \cite{bubu05, bdm93, hepi05}. Typical cases of shape functionals are the following.

\medskip{\it Integral functionals.} Given a right-hand side $f\in L^2(D)$, for every $\O\subset D$ we consider the solution $u_\O$ of the PDE \eqref{pde}, extended by zero on $D\setminus\O$. We may then consider the integral cost
$$\F(\O)=\int_D j\big(x,u_\O(x),\nabla u_\O(x)\big)\,dx,$$
where $j$ is a suitable integrand. For instance, an integration by parts in \eqref{pde} gives that the energy $E_{k,f}(\O)$ is an integral functional, with
$$j(x,s,z)=-\frac12 f(x)s.$$

\medskip{\it Spectral functionals.} For every domain $\O\subset D$ we consider the spectrum $\lambda(\O)$ of the Laplace operator $-\Delta$ on the Hilbert space $H^1_0(\O;D)$, with Robin condition $\frac{\partial u}{\partial n}+ku=0$ on the common boundary $\partial\O\cap\partial D$. Since the Lebesgue measure of $\O$ is finite, the operator $-\Delta$ has a compact resolvent and so its spectrum $\lambda(\O)$ consists of a sequence of eigenvalues $\lambda(\O)=\big(\lambda_j(\O)\big)_j$. The spectral cost functionals we may consider are of the form
$$\F(\O)=\Phi\big(\lambda(\O)\big),$$
for a suitable function $\Phi:\R^\N\to\overline{\R}$. For instance, taking $\Phi(\lambda)=\lambda_j$ we obtain
$$\F(\O)=\lambda_j(\O).$$
For an overview on spectral optimization problems we refer to \cite{but10, butve, deve}

\medskip The form of the optimization problems \eqref{shoptene} and \eqref{shoptgen} reminds the so-called {\it drop problems} (see for instance \cite{finn80,giusti81,gonzalez76,wente80} and references therein), where the cost functional $F(\O)$ involves the perimeter of $\O$ relative to $D$:
$$\F(\O)=\per(\O;D)+k\int_{\partial D}\ind_\O\,d\HH^{d-1}+\int_\O f(x)\,dx.$$

When $D$ is bounded we give a rather general existence theorem of optimal domains; assuming that the optimal domains are regular enough, we provide some necessary conditions of optimality describing the qualitative behaviour of the optimal sets. Another interesting situation occurs when $D=\R^d\setminus K$ where $K$ is the closure of a bounded Lipschitz domain. Also in this case a rather general existence result holds.

Finally we consider the case $\partial D$ unbounded and we provide some sufficient conditions for the existence of an optimal domain. We also provide some counterexamples showing that in general the existence of optimal domains may not occur.

In the paper, for simplicity, we consider the case $k=0$; the general case can be obtained by small modifications in the proofs.

%%%%%%%%%%%%%%%%%%%%%%%%%%%%%%%%%%%%%%%%%%%%%%%%%%
\section{Preliminaries}

\subsection{Capacity, quasi-open sets and quasi-continuous functions}
For an open set $\Dr\subset\R^d$, we denote with $H^1(\Dr)$ the Sobolev space, obtained is closure of the space $C^\infty(\R^d)$ with respect to the norm $$\|u\|_{H^1(\Dr)}:=\left(\int_{\Dr}|\nabla u|^2\,dx+\int_{\Dr}u^2\,dx\right)^{1/2}.$$
For a generic set $E\subset\R^d$ we define the capacity $\cp(E)$ as
$$\cp(E):=\min\Big\{\int_{\R^d}\big(|\nabla u|^2+u^2\big)\,dx:\ u\in H^1(\R^d),\ u\ge 1\ \hbox{in a neighbourhood of}\ E\Big\}.$$
We note that, $\cp(E)\ge |E|$ and so, the sets of zero capacity are also of Lebesgue measure zero. We will say that a property $\PP$ holds \emph{quasi-everywhere}, if $\PP$ hold for every point $x$, outside a set of capacity zero. 

\begin{deff}
We say that a set $\O\subset\R^d$ is \emph{quasi-open}, if for every $\eps>0$, there is an open set $\omega_\eps$ such that 
$$\cp(\omega_\eps)\le\eps\ \hbox{and the set}\ \O\cup\omega_\eps\ \hbox{is open}.$$
We say that a function $u:\Dr\to\R$ is \emph{quasi-continuous}, if for every $\eps>0$, there is an open set $\omega_\eps$ such that 
$$\cp(\omega_\eps)\le\eps\hbox{ and the restriction of $u$ on the set }\Dr\setminus\omega_\eps\ \hbox{is continuous}.$$
\end{deff}

It is well-known that a Sobolev function $u\in H^1(\Dr)$ has a quasi-continuous representative $\widetilde u$, which is unique up to a set of zero capacity. Moreover, in \cite{evgar} it was proved that quasi-every $x_0$ is a Lebesgue point for $u$ and the quasi-continuous representative $\widetilde u$ of $u$ can be pointwise characterized as 
$$\widetilde u(x_0)=\lim_{r\to0}\mean{B_r(x_0)}{u(x)\,dx}.$$
From now on, we will identify a Sobolev function $u$ with its quasi-continuous representative $\widetilde u$. 

By the definition of a quasi-open set and a quasi-continuous function, we note that for every Sobolev function $u\in H^1(\Dr)$, the level set $\{u>0\}$ is quasi-open. On the other hand, for each quasi-open set $\O\subset\R^d$, there is a Sobolev function $u\in H^1(\R^d)$ such that $\O=\{u>0\}$, up to a set of zero capacity.

We note that if the sequence $u_n\in H^1(\Dr)$ converges in $H^1(\Dr)$ to a function $u\in H^1(\Dr)$, then up to a subsequence $u_n(x)$ converges to $u(x)$ for quasi-every point $x\in\Dr$. Therefore, for every set $\O\subset\Dr$, the family of functions
$$H^1_0(\O;\Dr)=\Big\{u\in H^1(\Dr)\ :\ u=0\hbox{ q.e. on }D\setminus\O\Big\},$$
is a closed linear subspace of $H^1(\Dr)$. When $\Dr=\R^d$, we get simply $H^1_0(\O;\R^d)=H^1_0(\O)$, defined as the closure of $C^\infty_c(\O)$ with respect to the norm $\|\cdot\|_{H^1(\R^d)}$.\\

%%%%%%%%%%%%%%%
\subsection{Partial differential equations on quasi-open sets}
Let $\Dr\subset\R^d$ be an open set and let $\O\subset\Dr$ be a quasi-open set. For a given function $f\in L^2(\Dr)$, we say that $u$ is a solution of the partial differential equation (with mixed boundary conditions)
\be\label{pdembc}
-\Delta u=f\ \ \hbox{in}\ \ \O,\qquad \frac{\partial u}{\partial n}=0\ \ \hbox{on}\ \ \partial\Dr,\qquad u=0\ \ \hbox{on}\ \ \partial\O\cap\Dr,
\ee
if $u\in H^1_0(\O;\Dr)$ and 
$$\int_{\Dr}\nabla u\cdot\nabla v\,dx=\int_{\Dr}fv\,dx,\qquad\forall v\in H^1_0(\O;\Dr).$$

\begin{oss}\label{existspde}
Suppose that the connected open set $\Dr$ and the quasi-open $\O\subset\Dr$ are such that the inclusion $H^1_0(\O;\Dr)\hookrightarrow L^2(\Dr)$ is compact. Then we have:
\begin{itemize}
\item the first eigenvalue $\lambda_1(\O;\Dr)$ defined as
$$\lambda_1(\O;\Dr):=\min\Big\{\int_{\Dr}|\nabla v|^2\,dx:\ v\in H^1_0(\O;\Dr),\ \int_{\Dr}v^2\,dx=1\Big\},$$
is finite and strictly positive if $\O\neq\Dr$;
\item there is a unique minimizer $u_f\in H^1_0(\O;\Dr)$ of the functional
$$J_f(v)=\frac12\int_{\Dr}|\nabla v|^2\,dx-\int_{\Dr}vf\,dx,\qquad v\in H^1_0(\O;\Dr).$$
Writing the Euler-Lagrange equations for $u_f$, we get that it solves \eqref{pdembc}.
\end{itemize}
\end{oss}

We note that the inclusion $H^1_0(\O;\Dr)\hookrightarrow L^2(\Dr)$ is not always compact even if $\Dr$ is smooth and $\O$ is bounded. On the other hand it is well known that the compact inclusion $H^1_0(\O;\Dr)\hookrightarrow L^2(\Dr)$ occurs when:
\be\label{assdr}
\Dr\hbox{ is connected, uniformly Lipschitz and }|\O|<|\Dr|.
\ee
This covers for instance the following situations:
\begin{itemize}
\item $\Dr$ is bounded, $\partial\Dr$ is Lipschitz and $|\O|<|\Dr|$;
\item $\R^d\setminus\Dr$ is bounded, $\partial\Dr$ is Lipschitz and $|\O|<\infty$;
\item $\Dr$ is an unbounded convex open set and $|\O|<\infty$.
\end{itemize}

\begin{prop}\label{gencondprop}
Suppose that the open set $\Dr\subset\R^d$ and the quasi-open $\O\subset\Dr$ satisfy \eqref{assdr}. If $\delta<|\Dr|$ is such that $|\O|\le\delta$, then there is a constant $C>0$, depending on the dimension $d$, the constant $\delta$ and the box $\Dr$, such that
\begin{enumerate}[(i)]
\item the embedding $H^1_0(\O;\Dr)\hookrightarrow L^2(\Dr)$ is compact;
\item there is a constant $C>0$, depending only on the measure $|\O|$, such that 
\begin{align*}
\Big(\int_{\Dr} |u|^{2d/(d-2)}\,dx\Big)^{(d-2)/d}\le C\int_{\Dr}|\nabla u|^2\,dx
&\quad\forall u\in H^1_0(\O;\Dr),&\hbox{if}\ d\ge3;\\
\int_{\Dr}|u|^{\gamma}\,dx\le C\Big(\int_{\Dr}|u|^{\gamma-2}\,dx\Big)\Big(\int_{\Dr}|\nabla u|^2\,dx\Big)
&\quad\forall u\in H^1_0(\O;\Dr),\ \forall\gamma\ge2,& \hbox{if}\ d=2;
\end{align*}
\item for the first eigenvalue $\lambda_1(\O;\Dr)$, in any dimension $d\ge 2$, we have
$$\lambda_1(\O;\Dr)^{-1}\le C|\O|^{2/d}.$$
\end{enumerate}
\end{prop}

\begin{proof}
The claim {\it (i)} is standard and follows by the Lipschitz continuity of $\partial\Dr$, the claim {\it (ii)} and the fact that $|\O|<+\infty$. For {\it (ii)}, we notice that the condition $|\O|<|\Dr|$ and the connectedness of $\Dr$ provide the isoperimetric inequality
$$|\O|^{(d-1)/d}\le CP(\O;\Dr),$$
where $P(\O;\Dr)$ is the relative perimeter of $\O$ in $\Dr$ and $C$ is a constant depending on $\Dr$ and the measure of $\O$. Now {\it (ii)} follows by the inequality 
$$\left(\int_{\Dr}\vf^{d/(d-1)}\,dx\right)^{(d-1)/d}\le C\int_{\Dr}|\nabla \vf|\,dx,\qquad\forall \vf\in C^{\infty}(\Dr),$$
by replacing $\vf$ by $|u|^p$. The last claim {\it (iii)} follows by {\it (ii)} and the H\"older inequality.
\end{proof}

\begin{cor}
Suppose that the open set $\Dr\subset\R^d$ and the quasi-open $\O\subset\Dr$ satisfy \eqref{assdr}. Then for every function $f\in L^2(\Dr)$ the equation \eqref{pdembc} has a solution.
\end{cor}

The next result is well-known and we report it for the sake of completeness.

\begin{lemma}\label{infttybndflem}
Suppose that the open set $\Dr\subset\R^d$ and the quasi-open $\O\subset\Dr$ satisfy \eqref{assdr}. Let $f\in L^p(\R^d)$, where $p\in(d/2,+\infty]$, be a non-negative function and $u_f\in H^1_0(\O;\Dr)$ be the minimizer of $J_{f}$ in $H^1_0(\O;\Dr)$. If $\delta<|\Dr|$ is such that $|\O|\le \delta$, then we have a constant $C$, depending on the dimension $d$, the exponent $p$, the set $\Dr$ and the measure bound $\delta$, such that
$$\|u_f\|_{\infty}\le C\|f\|_{L^p}|\O|^{2/d-1/p}.$$
\end{lemma}

\begin{proof}
We set for simplicity $u:=u_f$. For every $t\in(0,\|u\|_\infty)$ and $\eps>0$, we consider the test function
$$u_{t,\eps}=u\wedge t +(u-t-\eps)^+.$$
Since $u_{t,\eps}\le u$ and $J_f(u)\le J_f(u_{t,\eps})$, we get
$$\frac12\int_{\Dr}|\nabla u|^2\,dx-\int_{\Dr}fu\,dx\le \frac12\int_{\Dr}|\nabla u_{t,\eps}|^2\,dx-\int_{\Dr}fu_{t,\eps}\,dx,$$
and after some calculations
$$\frac12\int_{\{t< u\le t+\eps\}}|\nabla u|^2\,dx\le \int_{\Dr} f\left(u-u_{t,\eps}\right)\,dx\le\eps\int_{\{u>t\}}f\,dx.$$
By the co-area formula we have
$$\int_{\{u=t\}}|\nabla u|\,d\HH^{d-1}\le 2\int_{\{u>t\}}f\,dx\le 2\|f\|_{L^p}|\{u>t\}|^{1/p'}.$$
Setting $\vf(t)=|\{u>t\}|$, for almost every $t$ we have 
\[\begin{split}
\vf'(t)&=-\int_{\{u=t\}}\frac{1}{|\nabla u|}\,d\HH^{d-1}\le-\Big(\int_{\{u=t\}}|\nabla u|\,d\HH^{d-1}\Big)^{-1}P(\{u>t\};\Dr)^2\\
&\le -\|f\|_{L^p}^{-1}\vf(t)^{-1+1/p}C_{iso}\vf(t)^{2(d-1)/d}=-\|f\|_{L^p}^{-1}C_{iso}\vf(t)^{(d-2)/d+1/p},
\end{split}\]
where $C_{iso}$ is the constant from the isoperimetric inequality in $\Dr$. Setting $\alpha=\frac{d-2}{d}+\frac1p$, we have that $\alpha<1$ and since the solution of the ODE
$$y'=-Ay^\alpha,\qquad y(0)=|\O|,$$
is given by
$$y(t)=\big(|\O|^{1-\alpha}-(1-\alpha)At\big)^{1/(1-\alpha)}.$$
Note that $\phi(t)\ge 0$, for every $t\ge0$, and $y(t)\ge \phi(t)$, if $\phi(t)>0$. Since $y$ vanishes in a finite time, there is some $t_{\max}$ such that $\phi(t)=0$, for every $t\ge t_{\max}$. Finally we obtain the estimate
$$\|u\|_{\infty}\le t_{\max}\le \frac{\|f\|_{L^p}|\O|^{2/d-1/p}}{C_{iso}\left(2/d-1/p\right)},$$
which concludes the proof.
\end{proof}

%%%%%%%%%%%%%%%
\subsection{Eigenfunctions and eigenvalues of the Laplacian with mixed boundary conditions}

In this subsection we suppose that $\Dr\subset\R^d$ and $\O\subset\Dr$ satisfy the condition \eqref{assdr}. Thus the \emph{resolvent} operator $R_\O:L^2(\Dr)\to L^2(\Dr)$, associating to each function $f\in L^2(\Dr)$ the solution $u_f$ of \eqref{pdembc}, is compact and self-adjoint. 

\begin{oss}\label{interp}
By the estimate of Lemma \ref{infttybndflem} we have that $R_\O$ can be extended to a continuous map $R_{\O}:L^p\to L^\infty$. On the other hand the resolvent is also a continuous map $R_\O:L^2\to L^{2d/(d-2)}$, for $d>2$, and $R_\O:L^2\to L^q$, for every $q>1$, for $d=2$. In dimension $d\ge 4$, a standard interpolation argument gives that $R_\O$ can be extended to a continuous operator
$$R_\O:L^q\to L^{\frac{q}{2}\frac{p-2}{p-q}\frac{2d}{d-2}},$$
for all $p>d/2$ and $q\in (2,p)$. This gives that $R_\O$ is a continuous operator
$$R_\O:L^q\to L^{q+\frac{4}{d-2}},$$
for all $q\in [2,d/2]$. In particular, for any $d\ge 2$, there is an entire number $n_d$ depending only on the dimension such that 
$$\left[R_\O\right]^{n_d}:L^2(\Dr)\to L^\infty(\Dr),$$
is a continuous operator.
\end{oss}

 Since the operator $R_\O:L^2\to L^2$ is compact, its spectrum is discrete. We define the spectrum of Laplacian on $\O$, with Neumann condition on $\partial\Dr$ and Dirichlet condition on $\Dr\cap\partial\O$, as the following sequence of inverse elements of the spectrum of $R_\O$.
$$\lambda_1(\O;\Dr)\le \lambda_2(\O;\Dr)\le\dots\le \lambda_k(\O;\Dr)\le\dots$$ 
We note that, for any $k\in\N$, the $k$th eigenvalue of the Laplacian can be variationally characterized as
$$\lambda_k(\O;\Dr)=\min_{S_k\subset H^1_0(\O;\Dr)}\ \max_{u\in S_k}\frac{\int_{\Dr}|\nabla u|^2\,dx}{\int_{\Dr}u^2\,dx},$$
where the minimum is taken over all $k$-dimensional subspaces $S_k\subset H^1_0(\O;\Dr)$. We note that there is a corresponding sequence of eigenfunctions $u_k\in H^1_0(\O;\Dr)$, forming a complete orthonormal sequence in $L^2(\O)$ and solving the equation
$$-\Delta u_k=\lambda_k(\O;\Dr)u_k\ \ \hbox{in}\ \ \O,\qquad \frac{\partial u_k}{\partial n}=0\ \ \hbox{on}\ \ \partial\Dr,\qquad u_k=0\ \ \hbox{on}\ \ \partial\O\cap\Dr.$$

\begin{prop}\label{unibndlbk}
Suppose that the open set $\Dr\subset\R^d$ and the quasi-open $\O\subset\Dr$ satisfy \eqref{assdr}. Then the eigenfunctions $u_k\in H^1_0(\O;\Dr)$ of the Laplace operator, with Dirichlet conditions on $\partial\O\cap\Dr$ and Neumann conditions on $\partial\Dr$, are bounded in $L^\infty(\Dr)$ by a constant that depends only on the dimension $d$, the eigenvalue $\lambda_k(\O;\Dr)$, the set $\Dr$ and the measure of $\O$.
\end{prop}

\begin{proof}
We note that 
$$R_\O[u_k]=\lambda_k(\O;\Dr)^{-1}u_k.$$
By Remark \ref{interp}, we have
$$R_\O^{n_d}[u_k]=\lambda_k(\O;\Dr)^{-n_d}u_k\in L^\infty(\Dr),$$
and, since $\int_{\Dr}u_k^2\,dx=1$, we have
$$\|u_k\|_{L^\infty}\le C\lambda_k(\O;\Dr)^{n_d},$$
where the constant $C$ depends on the measure of $\O$, $d$ and $\Dr$.
\end{proof}

%%%%%%%%%%%%%%%
\subsection{Energy and energy function}

Let $\O\subset\Dr$ be as above. We denote with $w_\O$ the solution of 
$$-\Delta w_\O=1\ \ \hbox{in}\ \ \O,\qquad \frac{\partial w_\O}{\partial n}=0\ \ \hbox{on}\ \ \partial\Dr,\qquad w_\O=0\ \ \hbox{on}\ \ \partial \O\cap\Dr,$$
and we will call it \emph{energy function} on $\O$, while the Dirichlet energy of $\O$ is defined as 
$$E_1(\O;\Dr):=-\frac12\int_{\Dr}w_\O\,dx.$$
Sometimes we will use the notation $R_\O(1)$ instead of $w_\O$. The properties of the energy function in a domain $\Dr$ are analogous to the properties of the energy function obtained solving the PDE with Dirichlet boundary condition on the whole $\partial\O$ (see \cite{bucve}). We summarize these properties in the following proposition.

\begin{prop}\label{proene}
For $\Dr\subset\R^d$ and $\O\subset\Dr$ as above, we have that the energy function $w_\O$ satisfies the following properties.
\begin{enumerate}[(a)]
\item $w_\O$ satisfies the bounds 
$$\int_{\Dr}|\nabla w_\O|^2\,dx\le 4\lambda_1(\O;\Dr)^{-1}|\O|\quad,\qquad\int_{\Dr} w_\O^2\,dx\le 4\lambda_1(\O;\Dr)^{-2}|\O|.$$
\item $w_\O$ is bounded and 
$$\|w_\O\|_{L^\infty}\le C(\Dr,|\O|),$$
where $C(\Dr,|\O|)$ is a constant depending only on $\Dr$ and the measure of $\O$.
\item $\Delta w_\O+\ind_{\{w_\O>0\}}\ge 0$ on $\Dr$, in sense of distributions. 
\item Every point of $\Dr$ is a Lebesgue point for $w_\O$.
\item $H^1_0(\O;\Dr)=H^1_0\big(\{w_\O>0\};\Dr\big)$. In particular, if $\O$ is a quasi-open set, then $\O=\{w_\O>0\}$ up to a set of zero capacity.
\end{enumerate} 
\end{prop}

\begin{proof}
The first claim {\it (a)} follows directly from the definition of an energy function. Claim {\it (b)} follows by Lemma \ref{infttybndflem}. The proofs of {\it (c)}, {\it (d)} and {\it (e)} are contained in \cite[Proposition 2.1]{bucve}.
\end{proof}

\begin{oss}
In particular, by condition {\it (c)} of Proposition \ref{proene} every quasi-open set $\O\subset\Dr$ of finite measure has a precise representative (up to a set of zero capacity) $\O=\{w_\O>0\}$.
\end{oss}

%%%%%%%%%%%%%%%%%%%%%%%%%%%%%%%%%%%%%%%%%%%%%%%%%%
\section{The $\gamma$-convergence}\label{sgamma}

In this section we endow the class of admissible domains $\O\subset\Dr$ with a convergence that will be very useful for our purposes. In the case of full Dirichlet conditions on $\partial\O$ this issue has been deeply studied under the name of $\gamma$-convergence, and we refer to \cite{bubu05} for all the related details. 

In what follows we assume that $\Dr\subset\R^d$ is a connected open set satisfying \eqref{assdr}.

\begin{deff}[$\gamma$-convergence]
Let $\O_n\subset\Dr$ be a sequence of quasi-open sets of finite measure and suppose that $\O_n\ne\Dr$. We say that $\O_n$ $\gamma$-converges to the quasi-open set $\O$, if the sequence of energy functions $w_{\O_n}\in H^1_0(\O_n;\Dr)$ converges 
strongly in $L^2(\Dr)$ to the energy function $w_\O\in H^1_0(\O;\Dr)$.
\end{deff}

The $\gamma$-convergence is a widely studied subject in shape optimization especially in the purely Dirichlet case $\Dr=\R^d$ and for domains $\O_n$ contained in a fixed ball $B\subset\R^d$. In this case various equivalent definitions were given to the $\gamma$-convergence:
\begin{itemize}
\item the convergence of the energy functions $w_{\O_n}\to w_\O$ in $L^2$;
\item the operator norm convergence of the resolvents $R_{\O_n}\to R_\O$ in $\LL(L^2)$;
\item the $\Gamma$-convergence of the functionals $F_{\O_n}\to F_\O$ in $L^2$.
\end{itemize}
If the constraint $\O_n\subset B$ is dropped, then the above definitions are no more equivalent even for $\Dr=\R^d$. As we will see below, the definition through the energy functions $w_{\O_n}$ is the strongest one and implies the other two. We will briefly recall the main results in the $\gamma$-convergence theory (for more details we refer to \cite{bubu05, dm93, tesi}). 

\begin{oss}\label{tighten}
Suppose that the sequence $\O_n\subset\Dr$ $\gamma$-converges to $\O$ and that $u_n\in H^1(\Dr)$ is a sequence such that 
$$|u_n|\le w_{\O_n}\qquad \hbox{and}\qquad \|u_n\|_{H^1(\Dr)}\le 1.$$
Then $u_n$ converges strongly in $L^2(\Dr)$ to some $u\in H^1(\Dr)$. This fact simply follows by the local compactness of the inclusion $H^1(\Dr)\hookrightarrow L^2(\Dr)$ and the tightness of $u_n$, which is due to the upper bound with a strongly converging sequence. 
\end{oss}

\begin{oss}\label{bounding}
Suppose that $\O\subset\Dr$ is a quasi-open set of finite measure and let $u\in H^1_0(\O;\Dr)$ be fixed. We denote with $A_{m,\O}(u)\in H^1_0(\O;\Dr)$ the unique minimizer of the functional
$$v\mapsto \int_{\Dr}\big(|\nabla v|^2+m|u-v|^2\big)\,dx,$$
in $H^1_0(\O;\Dr)$. Using $u$ to test the minimality of $A_{m,\O}(u)$ we get 
$$\|\nabla A_{m,\O}(u)\|_{L^2}\le \|\nabla u\|_{L^2} \qquad \hbox{and}\qquad \|A_{m,\O}(u)-u\|_{L^2}\le m^{-1/2}\|\nabla u\|_{L_2},$$
which gives the strong convergence of $A_{m,\O}(u)$ to $u$ in $L^2(\Dr)$ and also in $H^1(\Dr)$. The function $A_{m,\O}(u)$ satisfies the equation
$$-\Delta A_{m,\O}(u)+m A_{m,\O}(u) =m u\ \ \hbox{in}\ \ \O,\qquad \frac{\partial A_{m,\O}(u)}{\partial n}=0\ \ \hbox{on}\ \ \partial\Dr,\qquad A_{m,\O}(u)=0\ \ \hbox{on}\ \ \partial\O\cap\Dr,$$
and so, $A_{m,\O}$ can be extended to a linear operator on $L^2(\Dr)$. Moreover, $A_{m,\O}\le m R_\O$ in sense of operators on $L^2(\Dr)$ and thus, there is a number $N$ depending only on the dimension such that, after applying $N$ times the operator $A_{m,\O}$, we get
$$\|A_{m,\O}^N(u)\|_{L^\infty}\le C \|u\|_{L^2},$$
where $C$ is a constant depending on $m$, $d$, $\Dr$ and $|\O|$.
\end{oss}

\begin{prop}
Suppose that $\O_n\subset\Dr$ is a sequence of quasi-open sets, of uniformly bounded measure $|\O_n|\le C<|\Dr|$. Then the following are equivalent:
\begin{enumerate}[(i)]
\item the sequence $\O_n$ $\gamma$-converges to a quasi-open set $\O\subset\Dr$;
\item the sequence of energy functions $w_{\O_n}\in H^1(\O_n;\Dr)$ converges strongly in $H^1(\Dr)$ to the energy function $w_\O\in H^1(\O;\Dr)$;
\item for every sequence $f_n\in L^2(\Dr)$, converging weakly in $L^2$ to some $f\in L^2(\Dr)$, we have that $R_{\O_n}(f_n)$ converges strongly in $L^2(\Dr)$ to $R_\O(f)$;
\item the sequence of operators $R_{\O_n}\in\LL(L^2(\Dr))$ converges in the operator norm $\|\cdot\|_{\LL(L^2(\Dr))}$ to $R_\O\in\LL(L^2(\Dr))$. 
\end{enumerate}
\end{prop}

\begin{proof}
We first note that {\it (iii)}$\Leftrightarrow${\it (iv)} is standard and holds for a general sequence of compact operators on a Hilbert space. Thus it is sufficient to prove {\it (i)}$\Rightarrow${\it (ii)}$\Rightarrow${\it (iii)}$\Rightarrow${\it (i)}.

{\it (i)}$\Rightarrow${\it (ii)}. Due to the uniform bound of the Lebesgue measure of $\O_n$, we have a uniform bound on the norms $\|w_{\O_n}\|_{L^\infty}$ and so $w_{\O_n}$ converges to $w_\O$ also in $L^1(\Dr)$. Since using the equation we have 
$$\int_{\Dr}|\nabla w_{\O_n}|^2\,dx=\int_{\Dr}w_{\O_n}\,dx\to \int_{\Dr}w_{\O}\,dx=\int_{\Dr}|\nabla w_\O|^2\,dx,$$ 
which gives the strong convergence of the energy functions in $H^1(\Dr)$.

{\it (ii)}$\Rightarrow${\it (iii)}. We set for simplicity $\quad w_n=w_{\O_n},\quad w=w_\O\quad\hbox{and}\quad u_n=R_{\O_n}(f_n).$ We first note that $u_n$ converges strongly in $L^2(\Dr)$. In fact, by Remark \ref{bounding} and the maximum principle we get that for fixed $m>0$ the sequence $A_{m,\O_n}^{M+1}(u_n)$ is bounded (up to a constant depending on $m$ and $|\O_n|$) by $w_{\O_n}$. Thus, by Remark \ref{tighten}, it is a Cauchy sequence in $L^2(\Dr)$. Choosing $m$ large enough and observing that $\|u_n\|_{H^1(\Dr)}$ is bounded we get that $u_n$ is also a Cauchy sequence in $L^2(\Dr)$, converging strongly to some $u\in H^1(\Dr)$. We will now prove that $u=R_\O(f)$. Indeed, for every $\vf\in C^\infty_c(\Dr)$, we have
\[\begin{split}
\ds\int_{\Dr}u_n\vf\,dx&\ds=\int_{\Dr}\nabla w_n\cdot\nabla(u_n\vf)\,dx\\
&\ds=\int_{\Dr}\big(u_n\nabla w_n\cdot\nabla\vf-w_n\nabla u_n\cdot\nabla\vf\big)\,dx+\int_{\Dr}\nabla (w_n\vf)\cdot\nabla u_n\,dx\\
&\ds=\int_{\Dr}\big(u_n\nabla w_n\cdot\nabla\vf-w_n\nabla u_n\cdot\nabla\vf\big)\,dx+\int_{\Dr}w_n\vf f_n\,dx.
\end{split}\]
Passing to the limit as $n\to\infty$, we have
\be\label{Ralsosat}
\begin{array}{ll}
\ds\int_{\Dr}u\vf\,dx=\int_{\Dr}\big(u\nabla w\cdot\nabla\vf-w\nabla u\cdot\nabla\vf\big)\,dx+\int_{\Dr}w\vf f\,dx.
\end{array}
\ee
On the other hand, $R_\O(f)$ also satisfies \eqref{Ralsosat} and so, taking $v=u-R_\O(f)$, we have 
$$\int_{\Dr}v\vf\,dx=\int_{\Dr}\big(v\nabla w\cdot\nabla\vf-w\nabla v\cdot\nabla\vf\big)\,dx,\qquad\forall\vf\in C^\infty(\Dr),$$
which can be extended for test functions $\vf\in H^1(\R^d)$. Taking $v_t:=-t\vee v\wedge t$, as a test function, we get
$$\int_{\Dr}v_t^2\,dx\ds\le\int_{\Dr}\frac12\nabla w\cdot\nabla (v_t^2)-w|\nabla v_t|^2\,dx\le\frac12\int_{\Dr}v_t^2\,dx-\int_{\Dr}w|\nabla v_t|^2\,dx,$$
where we used that $\Delta w+1\ge 0$ on $\Dr$. In conclusion, we have
$$\frac12\int_{\Dr}v_t^2\,dx+\int_{\Dr}w|\nabla v_t|^2\,dx\le 0,$$
which gives $v_t=0$. Since $t>0$ is arbitrary, we obtain $u=R_\O(f)$, which concludes the proof of the implication {\it (ii)}$\Rightarrow${\it (iii)}.

{\it (iii)}$\Rightarrow${\it (i)}. Consider the sequence $f_n=\ind_{\O_n\cup\O}$. Since $f_n$ is bounded in $L^2(\Dr)$, we can suppose that, up to a subsequence $f_n$ converges weakly in $L^2(\Dr)$ to some $f\in L^2(\Dr)$. Moreover, we have that $0\le f\le 1$ and $f\ge \ind_\O$, since $f_n\ge \ind_\O$ for every $n\ge 1$. Thus, $f=1$ on $\O$ and $f_n=1$ on $\O_n$ and so we have that $w_{\O_n}=R_{\O_n}(f_n)$ converges strongly in $L^2(\Dr)$ to $w_\O=R_\O(f)$.
\end{proof}

Since the spectrum of compact operators is continuous with respect to the norm convergence, we have the following result.

\begin{cor}
Suppose that $\O_n\subset\Dr$ is a sequence of quasi-open sets, of uniformly bounded measure $|\O_n|\le C<|\Dr|$, which $\gamma$-converges to a quasi-open set $\O\subset\Dr$. Then, for every $k\in\N$ we have that the functional $\lambda_k(\cdot;\Dr)$ is continuous: 
$$\lim_{n\to\infty}\lambda_k(\O_n;\Dr)=\lambda_k(\O;\Dr).$$ 
\end{cor}

%%%%%%%%%%%%%%%
\subsection{$\gamma$-convergence of quasi-open sets and $\Gamma$-convergence of the associated functionals}

\begin{deff}\label{gtilde}
We say that the sequence of functionals $F_n:L^2(\Dr)\to[0,+\infty]$ $\Gamma$-converges in $L^2(\Dr)$ to the functional $F:L^2(\Dr)\to[0,+\infty]$, if
\begin{itemize}
\item[i)] for every $u_n\to u$ in $L^2(\Dr)$ we have
$$F(u)\le\liminf_n F_{n}(u_n);$$
\item[ii)] for every $u\in L^2(\Dr)$ there exists $u_n\to u$ in $L^2(\R^d)$ such that
$$F(u)=\lim_n F_{n}(u_n).$$
\end{itemize}
\end{deff}

To each quasi-open set $\O\subset\Dr$ we associate the functional $F_\O:L^2(\Dr)\to[0,+\infty]$ defined as
$$F_{\O}(u)=\begin{cases}
\ds\int_{\Dr}|\nabla u|^2\,dx&\hbox{if }u\in H^1_0(\O;\Dr);\\
+\infty&\hbox{otherwise}.
\end{cases}$$

\begin{prop}
Suppose that $\O_n\subset\Dr$, for $n\in\N$, and $\O\subset\Dr$ are quasi-open sets of uniformly bounded measure $|\O_n|\le C<|\Dr|$. Then $F_{\O_n}$ $\Gamma$-converges in $L^2(\Dr)$ to $F_\O$, if and only if, $R_{\O_n}$ converges strongly in $L^2(\Dr)$ to $R_\O$.
\end{prop}

\begin{proof}
Suppose first that $R_{\O_n}$ converges strongly in $L^2(\Dr)$ to $R_\O$. Let $u_n\in H^1_0(\O_n)$ be a sequence of uniformly bounded $H^1(\Dr)$ norm converging in $L^2(\Dr)$ to $u\in H^1(\Dr)$. Due to the identification $A_{m,\O}=R_\O(1+mR_\O)^{-1}$, we have that $A_{m,\O_n}$ also converges strongly to $A_{m,\O}$. Thus, we have 
\[\begin{split}
\|A_{m,\O_n}(u_n)-A_{m,\O}(u)\|_{L^2}&\le \|A_{m,\O_n}(u_n)-A_{m,\O_n}(u)\|_{L^2}+\|A_{m,\O_n}(u)-A_{m,\O}(u)\|_{L^2}\\
&\le \|R_{\O_n}\|_{\LL(L^2(\Dr))}\|u_n-u\|_{L^2}+\|A_{m,\O_n}(u)-A_{m,\O}(u)\|_{L^2},
\end{split}\]
which for fixed $m>0$ gives the convergence of $A_{m,\O_n}(u_n)$ to $A_{m,\O}(u)$. Now since $\|A_{m,\O_n}(u_n)-u_n\|_{L^2}\le m^{-1/2}\|u_n\|_{H^1(\Dr)}$, passing to the limit as $n\to\infty$, we get $\|A_{m,\O}(u)-u\|_{L^2}\le m^{-1/2} C$, which gives that $u\in H^1_0(\O;\Dr)$.

On the other hand, let $u\in H^1_0(\O;\Dr)$. Then we have that $A_{m,\O}(u)\to u$ in $H^1(\Dr)$, as $m\to\infty$. By the strong convergence of the resolvents we have $A_{m,\O_n}(u)\to A_{m,\O}(u)$ in $L^2(\Dr)$ for every fixed $m$ and $n\to\infty$. Using the equations for $A_{m,\O_n}(u)$ and $A_{m,\O}$ we have also that $\|A_{m,\O_n}(u)\|_{H^1(\Dr)}\to \|A_{m,\O}(u)\|_{H^1(\Dr)}$. Thus, it is sufficient to extract a diagonal sequence $A_{m_n,\O_n}(u)$ converging to $u$ in $H^1(\Dr)$.

Suppose now that $F_{\O_n}$ $\Gamma$-converges in $L^2(\Dr)$ to $F_\O$ and let $f\in L^2(\Dr)$ be a given function. Setting $u_n=R_{\O_n}(f)$, we get that $u_n$ is bounded in $H^1(\Dr)$ and so it converges in $L^2_{loc}$ to a function $u\in L^2_{loc}(\Dr)$. Moreover, using the equation for $u_n$ we have
$$\int_{\Dr}|\nabla ((1-\phi)u_n)|^2\,dx=\int_{\Dr}|\nabla \phi|^2 u_n^2\,dx+\int_{\Dr} u_n(1-\phi)^2 f\,dx,\qquad \forall\phi\in C^{0,1}_c(\R^d).$$
Now choosing $\phi$ to be $1$ in $B_R$, $0$ in $B_{2R}^c$ and harmonic in $B_{2R}\setminus B_R$, one has that
$$\int_{B_{2R}^c}u_n^2\,dx\le \lambda_1(\O_n;\Dr)^{-1}\int_{\Dr}|\nabla ((1-\phi)u_n)|^2\,dx\le\lambda_1(\O_n;\Dr)^{-1}\Big( \frac{C_d}{R^2}\|u_n\|_{L^2}^2+\|u_n\|_{L^2}\| f\ind_{B_{R}^c}\|_{L^2}\Big),$$
which gives that $u_n$ converges to $u\in H^1(\Dr)$ strongly in $L^2(\Dr)$. By the $\Gamma$-convergence of the functionals we have that $u\in H^1_0(\O;\Dr)$ and so it remains to prove that $u=R_\O(f)$. Indeed, for every $v\in H^1_0(\O;\Dr)$ there is a sequence $v_n\in H^1_0(\O_n;\Dr)$ such that
\[\begin{split}
\frac12\int_{\Dr}|\nabla v|^2\,dx-\int_{\Dr}vf\,dx&=\lim_{n\to\infty}\Big\{\frac12\int_{\Dr}|\nabla v_n|^2\,dx-\int_{\Dr}v_nf\,dx\Big\} \\ 
&\ge \liminf_{n\to\infty}\Big\{\frac12\int_{\Dr}|\nabla u_n|^2\,dx-\int_{\Dr}u_nf\,dx\Big\} \\
&\ge \liminf_{n\to\infty}\Big\{\frac12\int_{\Dr}|\nabla u|^2\,dx-\int_{\Dr}uf\,dx\Big\},
\end{split}\]
where we used the minimality of $u_n$ in the first inequality.
\end{proof}

\begin{prop}\label{omega0}
Suppose that $\O\subset\Dr$ and $\O_n\subset\Dr$, for $n\in\N$, are quasi-open sets, all contained in a quasi-open set of finite measure $\O_0\subset\Dr$ with $|\O_0|<|\Dr|$. Then the following are equivalent:
\begin{enumerate}[(i)]
\item $\O_n$ $\gamma$-converges to $\O$;
\item the sequence of resolvents $R_{\O_n}\in\LL(L^2(\Dr))$ converges in the operator norm to $R_\O\in\LL(L^2(\Dr))$;
\item the sequence of resolvents $R_{\O_n}\in\LL(L^2(\Dr))$ converges strongly in $L^2(\Dr)$ to $R_{\O}\in\LL(L^2(\Dr))$;
\item the sequence of functionals $F_{\O_n}$ $\Gamma$-converges in $L^2(\Dr)$ to $F_{\O}$. 
\end{enumerate}
\end{prop}

\begin{proof}
We already have that {\it (i)}$\Leftrightarrow${\it (ii)}$\Rightarrow${\it (iii)}$\Leftrightarrow${\it (iv)}. Thus it is sufficient to check that {\it (iv)}$\Rightarrow${\it (i)}. Indeed, let $w_n=w_{\O_n}$ be the sequence of energy functions of $\O_n$. By the uniform bound on $|\O_n|$ we have that $\|w_{n}\|_{H^1(\Dr)}\le C$ and, by the compact inclusion $H^1_0(\O_0;\Dr)\hookrightarrow L^2(\Dr)$ we can suppose that $w_n$ converges in $L^2(\Dr)$ to some $w\in H^1_0(\O_0;\Dr)$. By the $\Gamma$-convergence of $F_{\O_n}$ we have that $w\in H^1_0(\O;\Dr)$ and so it remains to prove that $w=w_\O$. Indeed, for every $v\in H^1_0(\O;\Dr)$ there is a sequence $v_n\in H^1_0(\O_n;\Dr)$ such that
\[\begin{split}
\frac12\int_{\Dr}|\nabla v|^2\,dx-\int_{\Dr}v\,dx&=\lim_{n\to\infty}\Big\{\frac12\int_{\Dr}|\nabla v_n|^2\,dx-\int_{\Dr}v_n\,dx\Big\} \\ 
&\ge \liminf_{n\to\infty}\Big\{\frac12\int_{\Dr}|\nabla w_n|^2\,dx-\int_{\Dr}w_n\,dx\Big\}\\
&\ge \liminf_{n\to\infty}\Big\{\frac12\int_{\Dr}|\nabla w|^2\,dx-\int_{\Dr}w\,dx\Big\},
\end{split}\]
which concludes the proof.
\end{proof}

\begin{oss}
We note that without the equiboundedness assumption $\O_n\subset\O_0$, the implication {\it (iii)}$\Rightarrow${\it (ii)} of Proposition \ref{omega0} may fail to be true. Take for instance $\Dr=\R^d$ and $\O_n=x_n+B_1$ with $|x_n|\to +\infty$. It is easy to see that $R_{\O_n}$ converges strongly in $L^2(\R^d)$ to zero, while 
$$\|R_{\O_n}\|_{\LL(L^2(\R^d);L^2(\R^d))}=\frac{1}{\lambda_1(B_1)}.$$
\end{oss}

%%%%%%%%%%%%%%%
\subsection{The weak-$\gamma$-convergence.}

\begin{deff}[weak-$\gamma$-convergence]
Let $\O_n\subset\Dr$ be a sequence of quasi-open sets of finite measure such that $|\O_n|<|\Dr|$. We say that $\O_n$ weak-$\gamma$-converges to the quasi-open set $\O$, if the sequence of energy functions $w_{\O_n}\in H^1_0(\O_n;\Dr)$ converges strongly in $L^2(\Dr)$ to a function $w\in H^1(\Dr)$ and $\O=\{w>0\}$ quasi-everywhere.
\end{deff}

By definition and the maximum principle $\O=\{w_\O>0\}$, we have that a $\gamma$-converging sequence $\O_n$ to $\O$ is also weak-$\gamma$-converging to $\O$. The converse is not true since an additional term may appear in the equation for the limit function $w$ (for a precise examples we refer to the book \cite{bubu05}). Nevertheless, one can obtain a sequence of quasi-open sets $\gamma$-converging to $\O$ simply by enlarging each of the sets $\O_n$. More precisely, the following proposition holds.

\begin{prop}\label{gwgprop}
Let $\O_n\subset\Dr$ be a sequence of quasi-open sets weak-$\gamma$-converging to a quasi-open set $\O\subset\Dr$. Then there is a sequence $\widetilde\O_n\subset\Dr$ such that $\O_n\subset\widetilde\O_n$ and $\widetilde\O_n$ $\gamma$-converges to $\,\O$. Moreover, if $\O_0\subset\Dr$ is a fixed quasi-open set such that $\O_n\subset \O_0$, for $n\ge1$, then the sequence $\widetilde\O_n$ can be chosen such that $\widetilde\O_n\subset\O_0$.
\end{prop}

In the case of full Dirichlet boundary conditions and $D$ bounded, the proof of Proposition \ref{gwgprop} can be found in \cite{bubu05}, \cite{but10} and \cite{buve}; the same proof can be repeated, step by step, to our more general setting.

We conclude this section with the following semi-continuity result, which can be found, for example, in \cite{but10} and \cite{buve}.

\begin{prop}\label{scres}
Suppose that the sequence of quasi-open sets $\O_n\subset\Dr$ weak-$\gamma$-converges to $\O$. Then we have:
$$|\O|\le\liminf_{n\to\infty}|\O_n|\qquad\hbox{and}\qquad\lambda_k(\O;\Dr)\le \liminf_{n\to\infty}\lambda_k(\O_n;\Dr),\quad\forall k\in\N.$$
\end{prop}

%%%%%%%%%%%%%%%%%%%%%%%%%%%%%%%%%%%%%%%%%%%%%%%%%%
\section{The spectral drop in a bounded domain}\label{sdbound}

In this section we consider the case when the box $\Dr$ is bounded. We obtain that in this case the optimal spectral drop exists for a very large class of shape cost functionals. More precisely, the following result holds.

\begin{teo}\label{exbound}
Let $\Dr\subset\R^d$ satisfy \eqref{assdr}. Suppose that the shape cost functional $\F$ on the quasi-open sets of $\Dr$ is such that:
\begin{itemize}
\item[1)] $\F$ is $\gamma$ lower semi-continuous, that is
$$\F(\O)\le\liminf_n\F(\O_n)\qquad\hbox{whenever}\quad\O_n\to_{\gamma}\O;$$
\item[2)] $\F$ is monotone decreasing with respect to the set inclusion, that is
$$\F(\O_1)\le\F(\O_2)\qquad\hbox{whenever}\quad\O_2\subset\O_1.$$
\end{itemize}
Then the shape optimization problem
\be\label{sopNDbox}
\min\Big\{\F(\O)\ :\ \O\subset\O_0,\ |\O|=1\Big\},
\ee
admits at least a solution.
\end{teo}

\begin{proof}
Suppose that $\O_n$ is a minimizing sequence for \eqref{sopNDbox}. Up to a subsequence, we may assume that $\O_n$ weak-$\gamma$-converges to a quasi-open set $\O\subset \Dr$. By Proposition \ref{gwgprop}, there are quasi open sets $\widetilde\O_n\subset\Dr$ such that the sequence $\widetilde\O_n$ $\gamma$-converges to $\O$ and $\O_n\subset\widetilde\O_n$. Then we have
$$\F(\O)\le\liminf_{n\to\infty}\F(\widetilde\O_n)\le\liminf_{n\to\infty}\F(\O_n),$$
and, on the other hand, by Proposition \ref{scres}, we have
$$|\O|\le\liminf_{n\to\infty}|\O_n|,$$
which concludes the proof since $\F$ is decreasing.
\end{proof}

\begin{cor}
Suppose that $F:\R^p\to\R$ is a lower-semi continuous function, increasing in each variable. Then the shape optimization problem 
$$\min\Big\{F\big(\lambda_{k_1}(\O;\Dr),\dots,\lambda_{k_p}(\O;\Dr)\big)\ :\ \O\subset D,\ |\O|=1\Big\},$$
has a solution.
\end{cor}

\begin{oss}\label{lb1touch}
We notice that, considering the shape cost functional $\F(\O)=\lambda_1(\O;\Dr)$ in \eqref{sopNDbox}, an optimal domain $\O$ must touch the boundary of $\Dr$. \emph{Precisely, if we suppose that $\Dr$ is smooth, then the measure $\HH^{d-1}\big(\partial\O\cap\partial\Dr\big)>0$.} Indeed, suppose that this is not the case, i.e. $\HH^{d-1}\big(\partial\O\cap\partial\Dr\big)=0$. Thus the trace of every function $u\in H^1_0(\O;\Dr)\subset H^1(\Dr)$ on the boundary $\partial\Dr$ is zero and so, since $\Dr$ is smooth, we have that $H^1_0(\O;\Dr)\subset H^1_0(\Dr)$, which in turn gives $H^1_0(\O;\Dr)=H^1_0(\O)$. 

Let now $u\in H^1_0(\O)$ be the first normalized eigenfunction on $\O$. Then a classical argument (see \cite[Chapter 6]{tesi}) gives that:

\begin{itemize}
\item the free boundary $\partial\O\cap\Dr$ is smooth and analytic (see \cite{brla});
\item there is a constant $\alpha>0$ such that 
$$|\nabla u|^2=\alpha\quad{on}\quad \partial\O\cap\Dr;$$
\item $u$ is Lipschitz continuous on $\R^d$ and $\O=\{u>0\}$. In particular, there is a constant $C>0$ such that 
\be\label{lipestbnd}
u(x)\le C\dist(x,\partial\O),\quad\hbox{for every}\quad x\in\O.
\ee
\end{itemize}
Up to translation of $\O$ in $\Dr$, we can assume that there is a point $x_0\in\partial\Dr\cap\partial\O$. Let $\nu$ be the external normal to $\partial\Dr$ in $x_0$ and let $\O_\eps:=(\eps\nu+\O)\cap\Dr$. Setting $u_\eps(x)=u(-e\nu+x)$ and applying \eqref{lipestbnd}, we get 
\be\label{changecondest}
\lambda_1(\O_\eps;\Dr)\le \frac{\int_{\O_\eps}|\nabla u_\eps|^2\,dx}{\int_{\O_\eps}u_\eps^2\,dx}\le \frac{\int_{\O}|\nabla u|^2\,dx}{1-\int_{\O\cap(-\eps\nu+\Dr^c)}u^2\,dx}\le \frac{\lambda_1(\O)}{1-C^2\eps^2|(\eps\nu+\O)\cap\Dr^c|}.
\ee
Now since $|(\eps\nu+\O)\cap\Dr^c|\to0$ as $\eps\to0$, for small enough $\eps$ we can find a smooth vector field $V_\eps\in C^{\infty}_c(\Dr;\R^d)$ such that the set $\widetilde \O_\eps:=(Id+V_\eps)(\O_\eps)$ satisfies 
$$|\widetilde\O_\eps|=|\O|=1\qquad\hbox{and}\qquad\lambda_1(\widetilde\O_\eps;\Dr)\le \lambda_1(\O_\eps;\Dr)-\frac\alpha2|(\eps\nu+\O)\cap\Dr^c|.$$ 
Together with \eqref{changecondest} this implies that for $\eps$ small enough $\lambda_1(\widetilde\O_\eps;\Dr)<\lambda_1(\O)$, which is a contradiction with the optimality of $\O$.
\end{oss}

\begin{oss}\label{lb1orth}
If we assume that $\Dr$ is smooth, then the boundary of an optimal domain $\O$ for \eqref{sopNDbox} intersects $\partial\Dr$ orthogonally. Indeed, by a smooth change of variables we may assume that $\partial\Dr$ is flat around the intersection point $x_0\in\partial\Dr\cap\partial\O$. We localize the problem in a small ball $B_r(x_0)$, in which we consider $\widetilde\O$ to be the union of $\O\cap B_r$ and its reflection with respect to $\partial\Dr$ as in Figure \ref{fig7}. Analogously we define $\widetilde u\in H^1(B_r(x_0))$ as the eigenfunction $u$ on $\O\cap B_r(x_0)$ and its reflection on the rest of $\widetilde\O$. Thus $\widetilde u$ is a solution of the free boundary problem
$$\min\Big\{J(v):\ v\in H^1(B_r(x_0)),\ v=\widetilde u\ \hbox{on}\ \partial B_r(x_0),\ |\{v>0\}|=|\{\widetilde u>0\}|\Big\},$$
where the functional $J:H^1(B_r(x_0))\to\R$ is defined as
$$J(v):=\frac{\frac12\int_{B_r(x_0)} |\nabla v|^2\,dx+\int_{\O\setminus B_r(x_0)} |\nabla u|^2\,dx}{\frac12\int_{B_r(x_0)} v^2\,dx+\int_{\O\setminus B_r(x_0)} u^2\,dx}.$$
Now by the same argument as in \cite{brla} the free boundary $\partial\{\widetilde u>0\}\cap B_r(x_0)=\partial\widetilde\O\cap B_r(x_0)$ is smooth and so, by the symmetry of $\widetilde\O$ we get that $\partial\widetilde\O$ is orthogonal to $\partial\Dr$.
\begin{figure}[h]
\includegraphics[scale=0.4]{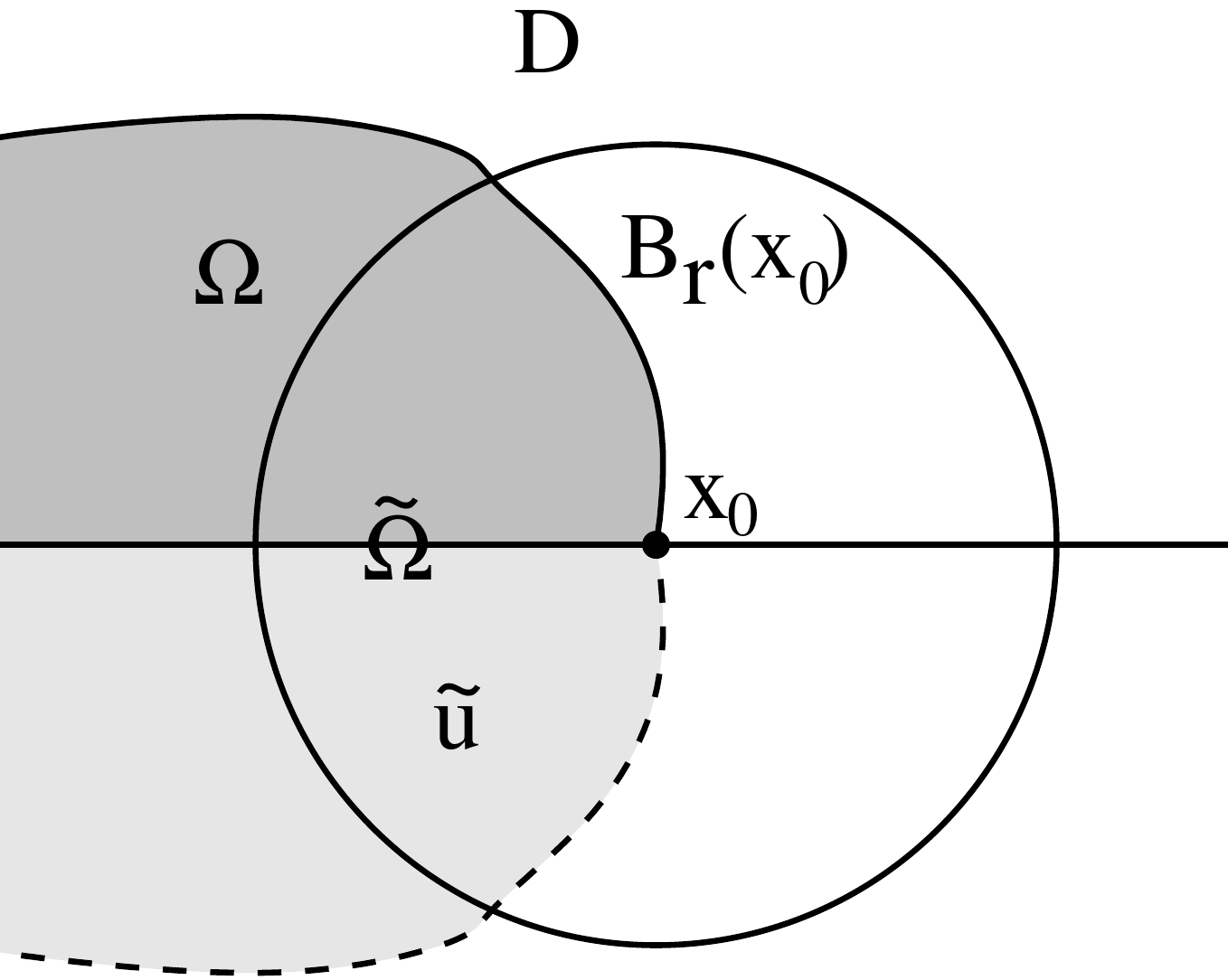}
\caption{The construction of the set $\widetilde\O$ which is a solution of a free boundary problem in a neighbourhood of $x_0\in\partial\O\cap\partial\Dr$.}
\label{fig7}
\end{figure}
\end{oss}

%%%%%%%%%%%%%%%%%%%%%%%%%%%%%%%%%%%%%%%%%%%%%%%%%%
\section{The spectral drop in unbounded domain}

In this section we discuss the existence of a solution to the shape optimization problem 
\be\label{sopNDlb1}
\min\Big\{\lambda_1(\O;\Dr):\ \O\subset\Dr,\ \O\ \hbox{quasi-open},\ |\O|=1\Big\},
\ee
in an unbounded domain $\Dr\subset\R^d$. The existence may fail since it might be convenient for a drop $\O\subset\Dr$ to escape at infinity as in the situation described in the following proposition.

\begin{prop}[Spectral drop in the complementary of a convex domain]\label{sopconcave}
Let $\Dr\subset\R^2$ be an open set whose complementary $\Dr^c$ is an unbounded closed strictly convex set. Then denoting by $H$ the half-space $\{(x,y)\in\R^2:\ y>0\}$ and by $B_+$ the half-ball $B_{\sqrt{2/\pi}}\cap H$, we have
$$\inf\Big\{\lambda_1(\O;\Dr):\ \O\subset\Dr,\ \O\ \hbox{quasi-open},\ |\O|=1\Big\}=\lambda_1(B_+;H),$$
and the infimum above is not attained and so the problem \eqref{sopNDlb1} does not have a solution.
\end{prop}

\begin{proof}
Let $\O\subset\Dr$ be a given quasi-open set of unit measure. We will first show that 
$$\lambda_1(B_+;H)<\lambda_1(\O;\Dr).$$
In order to do that consider the first normalized eigenfunction $u$ on $\O$ solving
$$-\Delta u=\lambda_1(\O;\Dr)u\ \ \hbox{in}\ \ \O,\qquad \frac{\partial u}{\partial \nu}=0\ \ \hbox{on}\ \ \partial \Dr,\qquad u=0\ \ \hbox{on}\ \ \partial \O\cap \Dr.$$
Consider the rearrangement $\widetilde u\in H^1_0(B_+;H)$ of $u$ (see Figure \ref{fig4}), defined through the equality 
$$\{\widetilde u>t\}= B_{\rho(t)}\cap H,\quad \hbox{where}\ \ \rho(t)\ \ \hbox{is such that}\ \ |B_{\rho(t)}|=2|\{u>t\}|.$$
\begin{figure}
\includegraphics[scale=0.4]{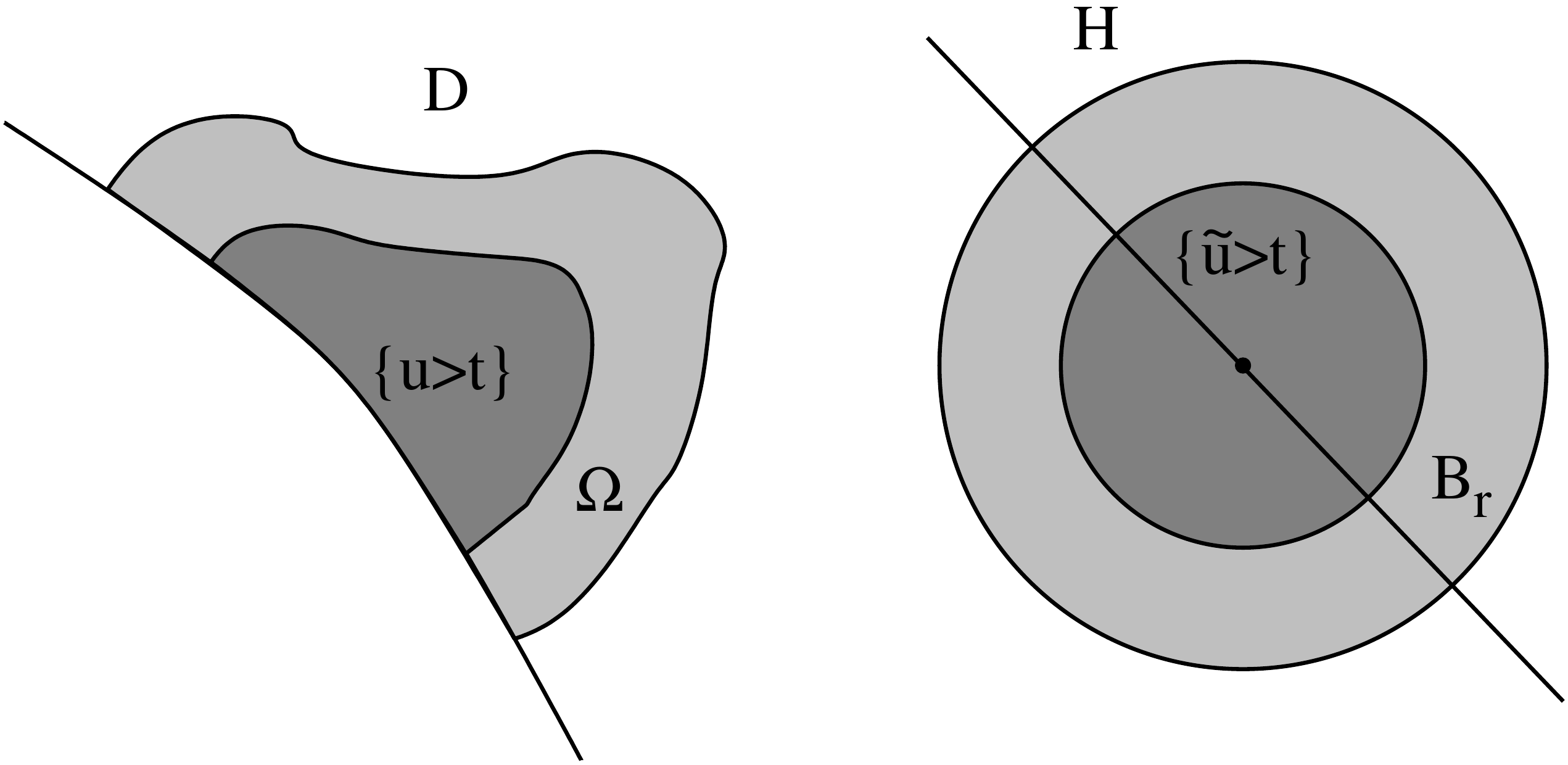}
\caption{A generic set $\O$ in the complementary of a (strictly) convex set (on the left) and a half-ball (on the right).}
\label{fig4}
\end{figure}
Then $\widetilde u$ is such that $|\nabla \widetilde u|=const$ on $B_{\rho(t)}$, for every $t>0$. Moreover, we have the isoperimetric inequality 
$$P(B_{\rho(t)};H)<P(\{u>t\};\Dr),\qquad\forall t>0.$$
Thus, setting $f(t)=|\{u>t\}|$ a standard co-area formula argument (see Example \ref{sssector}) gives
\[\begin{split}
\lambda_1(\O;\Dr)=\int_{\Dr}|\nabla u|^2\,dx
&\ge\int_{0}^{+\infty}\left(|f'(t)|^{-1}\,\HH^{1}\big(\{u=t\}\cap \Dr\big)^2\right)\,dt\\
&>\int_{0}^{+\infty}\left(|f'(t)|^{-1}\,\HH^{1}\big(\{\widetilde u=t\}\cap H\big)^2\right)\,dt\\
&=\int_{H}|\nabla \widetilde u|^2\,dx\ge \lambda_1(B_+;H).
\end{split}\]
Now it is sufficient to notice that choosing a sequence $x_n\in\partial\Dr$ such that $|x_n|\to+\infty$ one has that 
$$\dist_{\HH}\Big(\big(B_r\cap (-x_n+\Dr)\big)^c, B_+^c\Big)\to0,$$
where $\dist_{\HH}$ denotes the Hausdorff distance between closed sets. By \cite[Propostion 7.2.1]{bubu05} we have that
$$\lambda_1(B_r(x_n)\cap\Dr;\Dr)\to\lambda_1(B_+;H),$$
which proves the non-existence of optimal spectral drops in $\Dr$.
\end{proof}

We start our analysis of the spectral drop in an unbounded domain with three examples when optimal sets do exist. Namely, we consider the case when the domain $\Dr\subset\R^2$ is either a half space, an angular sector or a strip.

\begin{exam}[Spectral drop in a half-space]
Let $\Dr\subset\R^2$ be the half-plane 
$$\Dr=\Big\{(x,y)\in\R^2:\ y>0\Big\}.$$
Then the solution of \eqref{sopNDlb1} is given by the half ball $D\cap B_{\sqrt{2\pi}}$. Indeed, for any $\O\subset\Dr$, we have
$$\lambda_1(\O;\Dr)=\lambda_1\big(\O\cup\widetilde\O\cup(\partial\O\cap\partial\Dr)\big),$$
where $\widetilde\O$ is the reflection of $\O$
$$\widetilde\O=\Big\{(x,y)\in\R^2:\ (x,-y)\in\O\Big\}.$$
By the Faber-Krahn inequality we have that the optimal set of \eqref{sopNDlb1} is a half-ball centered on $\partial\Dr$ (see Figure \ref{fig1}).
\begin{figure}
\includegraphics[scale=0.4]{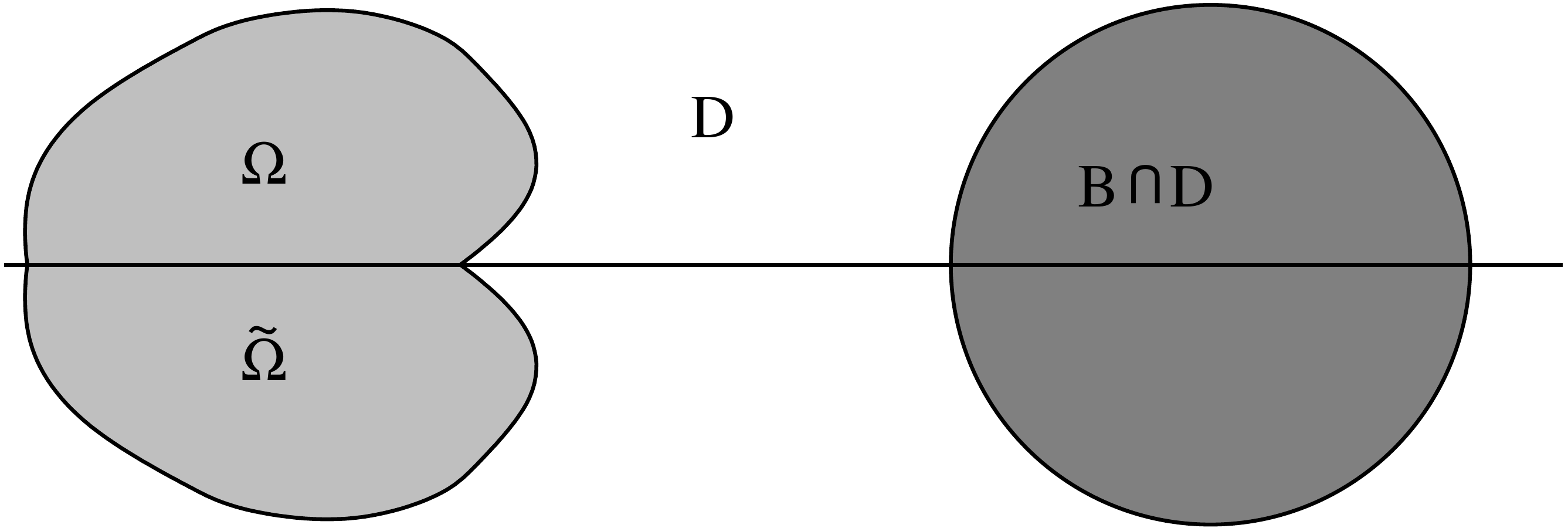}
\caption{A generic set $\O$ with its reflection $\widetilde\O$, on the left, and the optimal set $B\cap\Dr$, on the right.}
\label{fig1}
\end{figure}
\end{exam}

\begin{exam}[Spectral drop in an angular sector]\label{sssector}
Suppose now that $\Dr\subset\R^2$ is a sector 
$$\Dr=\Big\{(r\cos\theta, r\sin\theta)\in \R^2:\ r>0,\ \theta\in(-\alpha,\alpha)\Big\},$$
where $\alpha\in(0,\pi/2)$ is a given angle. We now prove that the unique solution of \eqref{sopNDlb1} is given by 
$$\Dr_{r_0}=\Big\{(r\cos\theta, r\sin\theta)\in \R^2:\ r_0>r>0,\ \theta\in(-\alpha,\alpha)\Big\},$$ 
where $r_0=\alpha^{-1/2}$. Indeed, let $\O\subset\Dr$ be a quasi-open set of unit measure and let $u$ be the first eigenfunction on $\O$. We considered the symmetrized function $\widetilde u\in H^1_0(\Dr_{r_0};\Dr)$ (see Figure \ref{fig2}), defined by 
$$\widetilde u(r,\theta)=\max\big\{t:\ |\{u>t\}|\le \alpha r^2\big\}.$$
\begin{figure}
\includegraphics[scale=0.4]{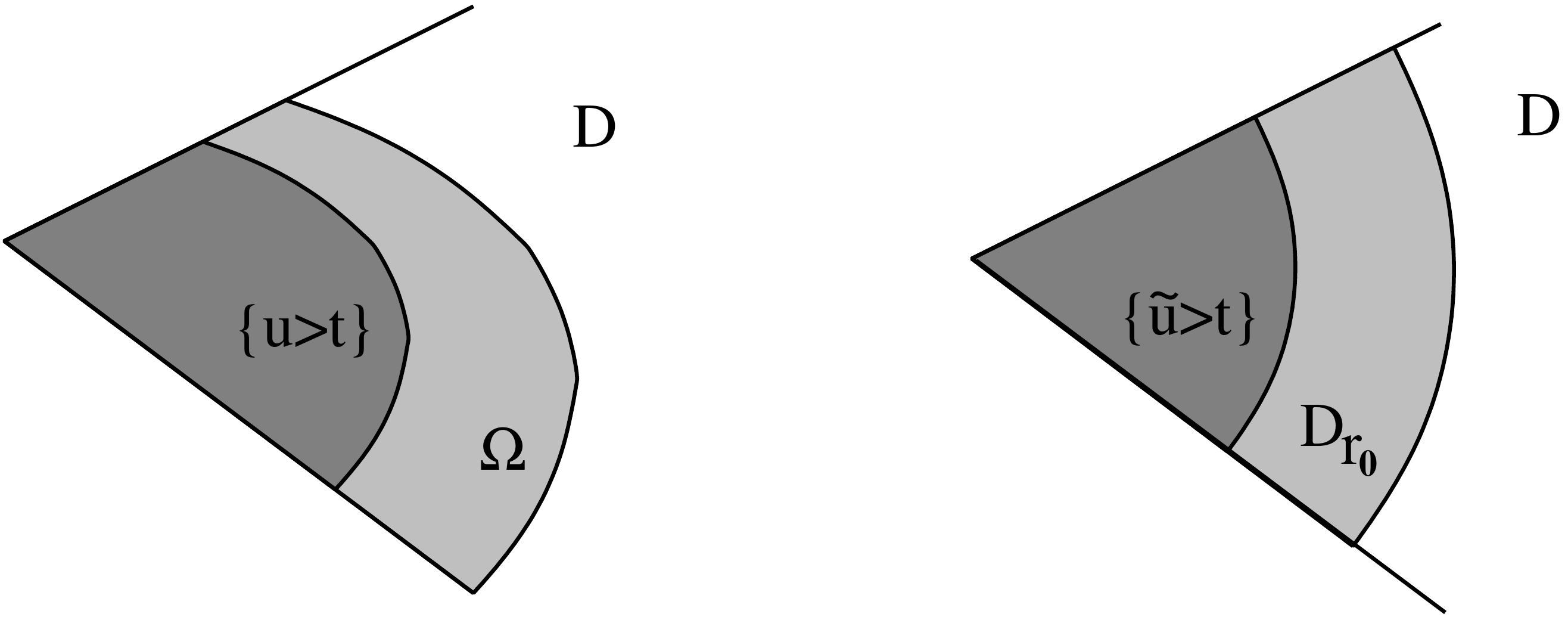}
\caption{A generic set $\O$ in the sector $\Dr$, on the left, and the optimal set $D_{r_0}$, on the right.}
\label{fig2}
\end{figure}
We now notice that $\int_{\Dr}\widetilde u^2\,dx=\int_\Dr u^2\,dx=1$ and
\[\begin{split}
\lambda_1(\O;\Dr)=\int_{\Dr}|\nabla u|^2\,dx&=\int_{0}^{+\infty}\Big(\int_{\{u=t\}}|\nabla u|\,d\HH^{1}\Big)\,dt\\
&\ge \int_0^{+\infty}\left(\Big(\int_{\{u=t\}}|\nabla u|^{-1}\,d\HH^{1}\Big)^{-1}\HH^{1}\big(\{u=t\}\cap\Dr\big)^2\right)\,dt\\
&=\int_0^{+\infty}\left(|f'(t)|^{-1}\,\HH^{1}\big(\{u=t\}\cap\Dr\big)^2\right)\,dt\\
&\ge\int_0^{+\infty}\left(|f'(t)|^{-1}\,\HH^{1}\big(\{\widetilde u=t\}\cap\Dr\big)^2\right)\,dt\\
&= \int_0^{+\infty}\left(\Big(\int_{\{\widetilde u=t\}}|\nabla \widetilde u|^{-1}\,d\HH^{1}\Big)^{-1}\HH^{1}\big(\{\widetilde u=t\}\cap\Dr\big)^2\right)\,dt\\
&=\int_0^{+\infty}\Big(\int_{\{\widetilde u=t\}}|\nabla \widetilde u|\,d\HH^{1}\Big)\,dt\\
&=\int_{\Dr}|\nabla \widetilde u|^2\,dx\ge \lambda_1(\Dr_{r_0};\Dr),
\end{split}\]
where $f(t)=|\{u>t\}|=|\{\widetilde u>t\}|$ and we used that $|\nabla \widetilde u|=const$ on $\{\widetilde u=t\}$ and that for every set $\O\subset\Dr$ the isoperimetric inequality $\HH^1(\Dr\cap\partial\Dr_r)\le \HH^1(\Dr\cap\partial\O)$ holds for $r=\sqrt{|\O|/\alpha}$.
\end{exam}

In the following example we note that the qualitative behaviour of the spectral drop may change as the measure of the drop changes.

\begin{exam}[Spectral drop in a strip]\label{ssstriscia}
Up to a coordinate change we may suppose that the strip is of the form $\Dr=\R\times (0,1)$. We consider for $c>0$ the problem
\be\label{sopstriscia}
\min\Big\{\lambda_1(\O;\Dr):\ \O\subset\Dr \ \hbox{ quasi-open},\ |\O|=c\Big\}.
\ee
We will prove that for $c$ small enough the optimal set for \eqref{ssstriscia} is a half-ball, while for $c$ large the optimal set is a rectangle $(0,c)\times(0,1)$.
\begin{itemize}
\item Let $c\le 2/\pi$. We first notice that if $|\O|\le 2/\pi$, then the isoperimetric inequality
$$P(\O;\Dr)^2\ge 2\pi|\O|,$$
holds with equality achieved when $\O$ is a half-ball centered on $\partial\Dr$. Thus, arguing as in Example \eqref{sssector} we get that the solution of \eqref{sopstriscia} is any half ball $B_{r}((0,y))$ with $r=\pi c/2\le 1$ and $y\in\R$. 

\item Let $c\ge2\sqrt2\pi$. We will prove that 
in this case the solution of \eqref{sopstriscia} is the rectangle $\O_c=(0,c)\times(0,1)$. Consider an open set $\O\subset\Dr$, of measure $|\O|=c$, such that 
$$l(t):=\HH^1(\{y=t\}\cap \O)>0,\qquad\forall t\in(a,b).$$
We will show that $\lambda_1(\O;\Dr)\ge \lambda_1(\O_c;\Dr)$. Setting $u\in H^1_0(\O;\Dr)$ to be the first normalized eigenfunction on $\O$, we have 
$$h(t):=\left(\int_0^1 u^2(x,t)\,dx\right)^{1/2}>0,\quad\forall t\in(0,1),\qquad\hbox{and}\qquad\int_0^1 h(t)^2\,dt=1.$$
Taking the derivative in $t$ we get
$$|h'(t)|=\frac{1}{h(t)}\left|\int_0^1 u_y(x,t) u(x,t)\,dx\right|\le \left(\int_0^1 u_y(x,t)^2\,dx\right)^{1/2}.$$
Now, using the decomposition $|\nabla u|^2=u_x^2+u_y^2$, we obtain
\[\begin{split}
\lambda_1(\O;\Dr)=\frac{\int_\O|\nabla u(x,y)|^2\,dx\,dy}{\int_\O u^2(x,y)\,dx\,dy}
&\ge\frac{\int_0^1\left(|h'(t)|^2+\int u_x^2(x,t)\,dx\right)\,dt}{\int_0^1 h^2(t)\,dt}\\
&\ge\frac{\int_0^1\left(|h'(t)|^2+\frac{\pi^2 h(t)^2}{l(t)^2}\right)\,dt}{\int_0^1 h^2(t)\,dt},
\end{split}\]
where the last inequality is due to the one-dimensional Faber-Krahn inequality
$$\frac{\int u_x^2(x,t)\,dx}{\int u^2(x,t)\,dx}\ge \lambda_1(\{y=t\}\cap\O)\ge \frac{\pi^2}{l(t)^2}.$$
Now we have
\be\label{minminmin}
\begin{split}
\lambda_1(\O;\Dr)\ge&\min\bigg\{\int_0^1\left(|h'(t)|^2+\frac{\pi^2 h(t)^2}{l(t)^2}\right)\,dt\\
&:\ h\in H^1(0,1),\ \|h\|_{L^2}=1,\ l\ge 0,\ \|l\|_{L^1}=c\bigg\}.
\end{split}\ee
Minimizing the right-hand side of \eqref{minminmin} first in $l$, we get 
$$\lambda_1(\O;\Dr)\ge\min\Big\{\int_0^1 |h'(t)|^2\,dt+\frac{\pi^2}{c^2}\left(\int_0^1 h(t)^{2/3}\,dt\right)^3: h\in H^1(a,b),\ \|h\|_{L^2}=1\Big\}.$$
Choosing $t_0\in(0,1)$ such that $\ds h(t_0)=\int_0^1 h^2(t)\,dt=1$, we get 
$$h^2(t)-1=h^2(t)-h^2(t_0)\le 2\int_0^1|h'(s)| h(s)\,ds\le 2\left(\int_0^1 |h'(s)|^2\,ds\right)^{1/2}.$$
Taking the square of the both sides and integrating for $t\in(0,1)$ we obtain the inequality
$$\int_0^1 h^4(t)\,dt\le 1+4\int_0^1 |h'(t)|^2\,dt,$$
with equality achieved for $h\equiv 1$. Thus we obtain 
\begin{align}
\lambda_1(\O;\Dr)\ge-\frac14+\min\Big\{\frac14\int_0^1 h(t)^4\,dt+\frac{\pi^2}{c^2}\left(\int_0^1 h(t)^{2/3}\,dt\right)^3:\ \|h\|_{L^2}=1\Big\}.\label{minlastmin}
\end{align}
Now by the Young inequality $a^p/p+b^q/q\ge ab$ with
$$\frac1p=\frac1{1+4\frac{\pi^2}{c^2}}\qquad\hbox{and}\qquad \frac1q=\frac{4\frac{\pi^2}{c^2}}{1+4\frac{\pi^2}{c^2}},$$
we obtain
\[\begin{split}
\frac14\int_0^1 h(t)^4\,dt+\frac{\pi^2}{c^2}\left(\int_0^1 h(t)^{2/3}\,dt\right)^3
&=\frac{1+4\frac{\pi^2}{c^2}}{4}\left(\frac1p\int_0^1 h(t)^4\,dt+\frac1q\left(\int_0^1 h(t)^{2/3}\,dt\right)^3\right)\\
&\ge\|h^{4/p}\|_{L^p}\|h^{2/q}\|_{L^{q/3}}\ge \|h^{(4/p+2/q)r}\|_{L^r},
\end{split}\]
where $\ds\frac1r=\frac1p+\frac3q$ and the equality holds when $h\equiv 1$. Since $c\ge 2\sqrt 2 \pi$, we have
$$2\le\frac{4+8\frac{\pi^2}{c^2}}{1+12\frac{\pi^2}{c^2}}=r\left(\frac4p+\frac2q\right),$$
and so, by the H\"older inequality we have $\|h^{(4/p+2/q)r}\|_{L^r}\ge 1$ with equality for $h\equiv 1$. Substituting in \eqref{minlastmin} we get 
$$\lambda_1(\O;\Dr)\ge \frac{\pi^2}{c^2}=\lambda_1(\O_c;\Dr).$$
\end{itemize}
\end{exam}

Proposition \ref{sopconcave} suggests that non-existence occurs when the spectral drop follows the boundary $\partial\Dr$ escaping at infinity. There are two particular cases of domains $\Dr$, for which the above situation can be avoided: \begin{itemize}
\item the case of an external domain $\Dr\subset\R^d$, i.e. a domain whose complementary $\Dr^c$ is bounded;
\item the case of an unbounded convex set $\Dr\subset\R^2$ in which a drop escaping at infinity would have less contact with the boundary $\partial\Dr$, which becomes flat at infinity. 
\end{itemize}
We treat these two cases in separate subsections. In the case of an external domain we are able to prove an existence result for a large class of spectral functionals $\F(\O)$, while in the case of a convex set we focus on the first eigenvalue $\lambda_1(\O;\Dr)$.

%%%%%%%%%%%%%%%
\subsection{Spectral drop in an external domain}

In this subsection we prove the existence of optimal sets for general spectral functionals $\F(\O)$ in a domain $\Dr\subset\R^d$, whose complementary $\Dr^c$ is a bounded set. The lack of the compact inclusion $H^1(\Dr)\hookrightarrow L^2(\Dr)$ adds significant difficulties to the existence argument since one has to study the qualitative behaviour of the minimizing sequences. Even in the simplest case $\Dr=\R^d$, in which the Neumann boundary $\partial\Dr$ vanishes, the question was solved only recently by Bucur \cite{bulbk} and Mazzoleni-Pratelli \cite{mp}. There are basically three different methods to deal with the lack of compactness:
\begin{itemize}
\item The first approach is based on a concentration-compactness argument for a minimizing sequence $\O_n$ of quasi-open sets in $\Dr$, as the one proved in \cite{buc00}. The \emph{compactness} situation leads straightforwardly to existence. The \emph{vanishing} case never occurs because this would give $\lambda_1(\O_n)\to+\infty$. The most delicate case is the \emph{dichotomy} when each set of the sequence is a union of two disjoint (and distant) quasi-open sets. At this point one notices that for spectral functionals one can run an induction argument on the number of eigenvalues that appear in the functional and their order. A crucial element of the proof is showing that the optimal sets remain bounded, thus in the case of dichotomy one can substitute the two distant quasi-open sets with optimal ones without overlapping. This approach was used in \cite{bulbk} in $\R^d$, in \cite{bubuve} in the case of an internal geometric obstacle and in \cite{bubuve2} in the case of Schr\"odinger potentials.
\item The second approach is to use the compactness of the inclusion $H^1(\Dr)\cap H^1_0(B_R)\hookrightarrow L^2(\Dr\cap B_R)$, for a ball $B_R\subset\R^d$ large enough, hence to prove the existence of an optimal domain among all quasi-open sets contained in $B_R$. Then prove that there is a \emph{uniform} bound on the diameter of the optimal sets. This approach was used in \cite{mp}.
\item The last approach consists in taking a minimizing sequence and modifying each of the domains, obtaining another minimizing sequence of uniformly bounded sets. One can choose a well behaving minimizing sequence by considering an auxiliary shape optimization problem in each of the quasi-open sets of the original minimizing sequence and then prove that the optimal sets have \emph{uniformly} bounded diameter. This is the method that was used in \cite{tesi} in $\R^d$ and the one we will use below in the case of general external domain $\Dr$.
\end{itemize}

As we saw above, the boundedness of the optimal sets is a fundamental step of the existence proof. For this, we will need the following notion of a \emph{shape subsolution}.

\begin{deff}
Let $\F$ be a functional on the family of quasi-open sets in $\Dr$. We say that $\O\subset\Dr$ is a \emph{shape subsolution} (or just \emph{subsolution}) for $\F$ if it satisfies
\be\label{subineq}
\F(\O)\le\F(\omega),\quad\hbox{for every quasi-open set}\quad\omega\subset\O.
\ee
We say that $\O$ is a \emph{local subsolution} if \eqref{subineq} holds for quasi-open sets $\omega\subset\O$ such that $\O\setminus\omega$ is contained in a ball of radius less than some fixed $\eps>0$.
\end{deff}

\begin{lemma}
Suppose that the quasi-open set $\O\subset\Dr$ is a subsolution for the functional $\F=F\big(\lambda_{1}(\O;\Dr),\dots,\lambda_{k}(\O;\Dr)\big)+\Lambda|\O|$, where $F:\R^k\to\R$ is a locally Lipschitz continuous function. Then $\O$ is a local subsolution for the functional $\G(\O)=E_1(\O;\Dr)+\Lambda'|\O|$, where the constants $\Lambda'$ and $\eps$ depend on $\Dr$, $F$, $\Lambda$, $\lambda_{k_1}(\O;\Dr),\dots,\lambda_{k_p}(\O;\Dr)$ and $|\O|$. 
\end{lemma}

\begin{proof}
Let $\omega\subset\O$ be a quasi-open set. By \cite[Lemma 3]{bulbk} or \cite[Lemma 3.7.7]{tesi}, we have the estimate
\be\label{dorinest}
\lambda_k(\O;\Dr)^{-1}-\lambda_k(\omega;\Dr)^{-1}\le C\big(E_1(\omega;\Dr)-E_1(\O;\Dr)\big),
\ee
where $C$ is a constant depending on the sum $\sum_{i=1}^k\|u_i\|_\infty$.
By the sub-optimality of $\O$ and the estimate \eqref{dorinest} we have
\[\begin{split}
\Lambda|\O\setminus\omega|&\le F\big(\lambda_{1}(\omega;\Dr),\dots,\lambda_{k}(\omega;\Dr)\big)-F\big(\lambda_{1}(\O;\Dr),\dots,\lambda_{k}(\O;\Dr)\big)\\
&\le L\sum_{i=1}^k\big(\lambda_{i}(\omega;\Dr)-\lambda_{i}(\O;\Dr)\big)\\
&= L\sum_{i=1}^k\lambda_{i}(\omega;\Dr)\lambda_{i}(\O;\Dr)\big(\lambda_{i}(\O;\Dr)^{-1}-\lambda_{i}(\omega;\Dr)^{-1}\big)\\
&\le LC\big(E_1(\omega;\Dr)-E_1(\O;\Dr)\big)\sum_{i=1}^k\lambda_{i}(\omega;\Dr)\lambda_{i}(\O;\Dr),
\end{split}\]
where $L$ is the Lipschitz constant of $F$ and $C$ is the constant from \eqref{dorinest}.
\end{proof}

The following lemma is classical and a variant was first proved by Alt and Caffarelli in \cite{altcaf}, for a precise statement we refer to \cite{bulbk} and \cite{bucve}.

\begin{lemma}\label{altcaflem}
Suppose that the quasi-open set $\O\subset\Dr$ is a local subsolution for the functional $\G(\O)=E_1(\O;\Dr)+\Lambda|\O|$. Then there are constants $r_0>0$ and $c>0$, depending on $\Lambda$ and $\eps$, such that the following implication holds
$$\Big(\mean{B_r(x_0)}{w_\O\,dx}\le cr\Big)\ \Rightarrow\ \Big(w_\O=0\ \hbox{in}\ B_{r/4}(x_0)\Big),$$
for every $x_0\in\Dr$ and $0<r\le r_0$ such that $B_r(x_0)\subset \Dr$. 
\end{lemma}

The following Lemma was proved in \cite{bucve} in the case $\Dr=\R^d$.

\begin{lemma}\label{unifdiambnd}
Suppose that the quasi-open set $\O\subset\Dr$ is a local subsolution for the functional
$$\G(\O)=E_1(\O;\Dr)+\Lambda|\O|.$$
Then $\O$ is a bounded set. Moreover, for $r>0$ small enough the set
$$\O_r:=\O\cap\big\{x\in\Dr\ :\ \dist(x,\partial\Dr)>2r\big\},$$ 
can be covered by $N_r$ balls of radius $r$, where the number of balls $N_r$ depends on $\eps$, $\Lambda$ and $\Dr$. 
\end{lemma}

\begin{proof}
We construct a sequence $(x_n)_{n\ge 1}$ as follows: choose $x_1\in\O_r$; given $x_1,\dots,x_{n_1}$, we choose $x_n\in \O_r\setminus\left(\bigcup_{i=1}^{n-1}B_{2r}(x_i)\right)$. We notice that, by construction $w_\O(x_n)>0$ and that the balls $B_{r}(x_i)$ are pairwise disjoint for $i=1,\dots,n$. Thus, by Lemma \ref{altcaflem}, we have that 
$$\int_{\Dr}w_\O\,dx\le \sum_{i=1}^n \int_{B_r(x_i)}w_\O\,dx\le n c \omega_d r^{d+1},$$
and so, if $N$ is the largest integer such that
$$N\le \frac{1}{c\omega_d r^{d+1}}\int_{\Dr}w_\O\,dx,$$
the sequence $x_n$ can have at most $N$ elements.
\end{proof}

We are now in position to prove our main existence result in an external domain $\Dr\subset\R^d$.

\begin{teo}\label{exbound}
Assume that $D$ is an external domain, that is an open set satisfying \eqref{assdr} with bounded complementary $\Dr^c$, and that the function $F:\R^k\to\R$ is increasing and Lipschitz continuous. Then the shape optimization problem
\be\label{sopNDext}
\min\Big\{F\big(\lambda_{1}(\O;\Dr),\dots,\lambda_{k}(\O;\Dr)\big)+\Lambda|\O|\ :\ \O\hbox{ quasi-open, }\O\subset D\Big\},
\ee
has a solution. Moreover, every solution of \eqref{sopNDext} is a bounded set.
\end{teo}

\begin{proof}
Let $\O_n$ be a minimizing sequence for \eqref{sopNDext}. Since each of the quasi-open sets $\O_n$ has finite measure, we have that weak-$\gamma$-convergence is compact in $\O_n$ and so, the shape optimization problem
$$\min\Big\{F\big(\lambda_{1}(\O;\Dr),\dots,\lambda_{k}(\O;\Dr)\big)+\Lambda|\O|\ :\ \O\hbox{ quasi-open, }\O\subset\O_n\Big\},$$
has at least one solution $\widetilde\O_n$. Since 
$$F\big(\lambda_{1}(\widetilde\O_n;\Dr),\dots,\lambda_{k}(\widetilde\O_n;\Dr)\big)+\Lambda|\widetilde\O_n|\le F\big(\lambda_{1}(\O_n;\Dr),\dots,\lambda_{k}(\O_n;\Dr)\big)+\Lambda|\O_n|,$$
we have that the sequence $\widetilde\O_n$ is also minimizing. Moreover, each of the sets $\widetilde\O_n$ is a subsolution for $\F$ and so, a local subsolution for $\G(\O)=E_1(\O;\Dr)+\Lambda'|\O|$. By Lemma \ref{unifdiambnd}, we can cover the set $\widetilde\O_n\setminus(\Dr^c+B_1)$ by a finite number of balls of radius, which does not depend on $n$. Setting $A_n$ to be the open set obtained as a union of these balls, we can translate the parts of $\widetilde\O_n$ contained in the different connected components of $A_n$ obtaining a new set, which we still denote by $\widetilde\O_n$ and which has the same measure and spectrum. Moreover, we now have that $\widetilde\O_n\subset B_R$, for some $R>0$ large enough. Again, by the compactness of the weak-$\gamma$-convergence in $B_R$, we have that up to a subsequence $\widetilde\O_n$ weak-$\gamma$-converges to a set $\widetilde\O\subset\Dr\cap B_R$. By the semi-continuity of $\lambda_k(\cdot;\Dr)$ and the Lebesgue measure (Proposition \ref{scres}), we have 
$$F\big(\lambda_{1}(\widetilde\O;\Dr),\dots,\lambda_{k}(\widetilde\O;\Dr)\big)+\Lambda|\widetilde\O|\le\liminf_{n\to\infty}\Big\{F\big(\lambda_{1}(\widetilde\O_n;\Dr),\dots,\lambda_{k}(\widetilde\O_n;\Dr)\big)+\Lambda|\widetilde\O_n|\Big\},$$
which proves that $\widetilde\O$ is a solution of \eqref{sopNDext}.
\end{proof}

\begin{oss}
By arguments similar to the ones used in Remarks \ref{lb1touch} and \ref{lb1orth} we obtain that the optimal domain $\O^\ast$ for the functional $\F(\O)=\lambda_1(\O;\Dr)$ satisfies the following properties:
\begin{itemize}
\item the free boundary $\Dr\cap \partial\O^\ast$ is smooth;
\item $\O^\ast$ must touch the boundary $\partial\Dr$;
\item if $\Dr$ is smooth, then the boundary of $\O^\ast$ intersects $\partial\Dr$ orthogonally. 
\end{itemize}
\end{oss}

%%%%%%%%%%%%%%%
\subsection{A spectral drop in unbounded convex plane domains}

In this subsection we consider the case when $\Dr$ is an unbounded convex domain in $\R^2$. We note that the unbounded convex sets in $\R^2$ can be reduced to the following types:
\begin{itemize}
\item a strip $\Dr=(a,b)\times \R$;
\item an epigraph of a convex function $\vf:\R\to\R$;
\item an epigraph of a convex function $\vf:(a,b)\to\R$.
\end{itemize}

In order to prove the existence of an optimal set we argue as in the case of external domains and we consider the following penalized version of the shape optimization problem:
\be\label{sopNDpen}
\min\Big\{\lambda_1(\O;\Dr)+\Lambda|\O|\ :\ \O\subset\Dr,\ \O\hbox{ quasi-open}\Big\}.
\ee

In what follows we will concentrate our attention to the third case when the convex domain is an epigraph of a convex function defined on the entire line $\R$. 

Since we are in two dimensions the uniform bound on the minimizing sequence is easier to achieve through an estimate on the perimeter $P(\O;\Dr)$. The following result was proved in \cite{bulbk}.

\begin{lemma}
Suppose that the quasi-open set $\O\subset\Dr$ is a subsolution for the functional $\F(\O)=\lambda_1(\O)+\Lambda|\O|$. Then $\O$ has finite perimeter and 
$$P(\O;\Dr)\le \Lambda^{-1/2}\lambda_1(\O;\Dr)|\O|^{1/2}.$$
\end{lemma}

\begin{teo}
Let $\vf:\R\to\R$ be a convex function and let $\Dr=\{(x,y)\in\R^2:\ y>\vf(x)\}$. Then there exists a solution of the problem \eqref{sopNDpen}. Moreover, every solution $\O$ of \eqref{sopNDlb1} is a bounded open set of finite perimeter whose boundary is locally a graph of an analytic function, intersecting the boundary $\partial\Dr$ orthogonally. 
\end{teo}

\begin{proof}
Let $\O_n\subset\Dr$ be a minimizing sequence for \eqref{sopNDpen}. For every $\O_n$ we consider a solution $\widetilde\O_n$ of the problem 
$$\min\Big\{\lambda_1(\O;\Dr)+\Lambda|\O|:\ \O\subset\O_n,\ \O\ \hbox{quasi-open}\Big\}.$$
We first notice that $\widetilde\O_n$ is also a minimizing sequence for \eqref{sopNDpen}. Since each of the sets $\widetilde\O_n$ is a subsolution for the functional $\F(\O)=\lambda_1(\O)+\Lambda|\O|$ we have that the bound
$$P(\widetilde\O_n;\Dr)\le \Lambda^{-1/2}\lambda_1(\widetilde\O_n;\Dr)|\widetilde\O_n|^{1/2},$$
holds or every $n\in\N$. Thus, there is a universal bound on the diameter $\hbox{diam}(\widetilde\O_n)\le R<+\infty$, for all $n\in\N$. Thus, for every $\widetilde\O_n$, there is a ball $B_{R}(x_n)$ such that $\widetilde\O_n\subset B_R(x_n)$. We now consider, for every $n\in\N$, a solution $\O_n^\ast$ of the problem
$$\min\Big\{\lambda_1(\O;\Dr)+\Lambda|\O|\ :\ \O\subset B_R(x_n)\cap\Dr,\ \O\ \hbox{quasi-open}\Big\}.$$
Notice that $\O_n^\ast$ is still a minimizing sequence for \eqref{sopNDpen} and has uniformly bounded perimeter and diameter. If the sequence $x_n$ is bounded, then $\widetilde\O_n$ are all contained in a large ball $B_{R^\ast}$, which by the compactness of the weak-$\gamma$-convergence and the lower semi-continuity of the functional, gives the existence of an optimal set.

Suppose, by absurd, that (up to a subsequence) we have that $|x_n|\to+\infty$. We notice that up to translating the balls, which are entirely contained in $\Dr$ and enlarging the fixed radius $R$, we can suppose that $x_n\in\partial\Dr$, for every $n\in\N$. Now since the boundary of an unbounded convex set is getting flat at infinity, we have that there is a sequence of half-spaces $H_n\subset\R^2$ such that $H_n\cap B_R(x_n)\subset\Dr\cap B_R(x_n)$ for all $n\in\N$, and
$$\dist_{\HH}\big(B_R(x_n)\cap\partial H_n, B_R(x_n)\cap\partial\Dr\big)\xrightarrow[n\to\infty]{}0,$$
where $\dist_{\HH}$ is the Hausdorff distance between compact sets in $\R^2$.

Let now $u_n\in H^1_0(\O_n^\ast;\Dr)$ be the first normalized eigenfunction on $\O_n^\ast$ with mixed boundary conditions
$$-\Delta u_n=\lambda_1(\O_n^\ast;\Dr)u_n\ \hbox{ in }\ \O_n^\ast,\qquad \frac{\partial u_n}{\partial \nu}=0\ \ \hbox{on}\ \ \partial \Dr,\qquad u_n=0\ \ \hbox{on}\ \ \partial \O_n^\ast\cap\Dr.$$
Consider the quasi-open set $\omega_n^\ast=\O_n^\ast\cap H_n$. Then we have 
$$\lambda_1(\omega_n^\ast;H_n)\le \frac{\int_{H_n}|\nabla u_n|^2\,dx}{\int_{H_n} u_n^2\,dx}\le \frac{\lambda_1(\O_n^\ast;\Dr)}{1-\|u_n\|_\infty^2|\Dr\setminus H_n|}\le \frac{\lambda_1(\O_n^\ast;\Dr)}{1-C|\Dr\setminus H_n|},$$
where the last inequality is due to the uniform bound on the infinity norm of the eigenfunctions proved in Proposition \ref{unibndlbk}.

Let now $H=\{(x,y):\ y >0\}$, $B_{r_n}$ be the ball of measure $|\omega_n^\ast|$ centered at the origin and $r=\lim_{n\to\infty}r_n$. Then we have
\[\begin{split}
\lambda_1(B_{r}\cap H;H)+\Lambda |B_{r}\cap H|&=\lim_{n\to\infty}\Big\{\lambda_1(B_{r_n}\cap H;H)+\Lambda|B_{r_n}\cap H|\Big\}\\
&\le\liminf_{n\to\infty}\Big\{\lambda_1(\omega_n^\ast;H_n)+\Lambda |\omega_n^\ast|\Big\}\\
&\le\liminf_{n\to\infty}\Big\{\lambda_1(\O_n^\ast;\Dr)+\Lambda |\O_n^\ast|\Big\}.
\end{split}\]
In order to prove that the minimizing sequence $\O_n^\ast$ cannot escape at infinity, it is sufficient to show that 
$$\lambda_1(B_{\rho}\cap \Dr;\Dr)+\Lambda |B_{\rho}\cap \Dr|\le \lambda_1(B_{r}\cap H;H)+\Lambda |B_{r}\cap H|,$$
where we assume that $0\in\partial\Dr$ is a point where $\partial\Dr$ is not flat and choose $\rho>r$ such that $|B_\rho\cap\Dr|=|B_r\cap H|$. We consider the first normalized eigenfunction $u\in H^1_0(B_\rho\cap H;H)$ on the half-ball 
$$-\Delta u=\lambda_1(B_r\cap H;H)u\ \ \hbox{in}\ \ B_r\cap H,\qquad \frac{\partial u}{\partial \nu}=0\ \ \hbox{on}\ \ \partial H,\qquad u=0\ \ \hbox{on}\ \ \partial B_r\cap H,$$
and we consider the rearrangement $\widetilde u\in H^1_0(B_\rho\cap\Dr;\Dr)$ of $u$ (see Figure \ref{fig3}) defined as:
$$\{\widetilde u>t\}=B_{\rho(t)}\cap \Dr,\ \ \hbox{where}\ \ \rho(t)>0\ \ \hbox{is such that}\ \ |B_{\rho(t)}\cap \Dr|=|\{u>t\}|.$$
\begin{figure}
\includegraphics[scale=0.4]{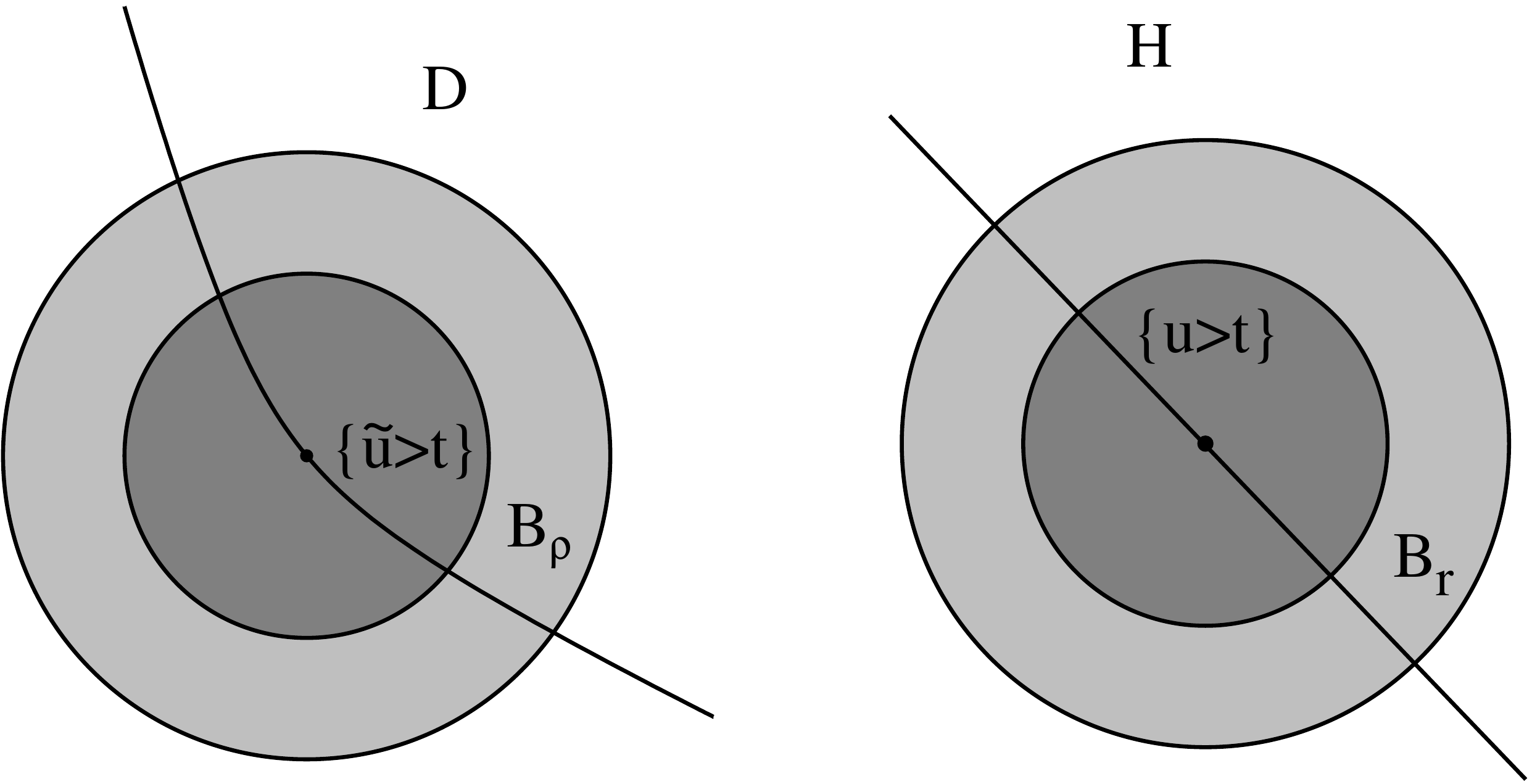}
\caption{A half-ball obtained as limit of a minimizing sequence escaping at infinity (on the right) and a competitor with circular level sets (on the left).}
\label{fig3}
\end{figure}
We notice that $\widetilde u$ is constant on each circle $\partial B_{\rho(t)}\cap\Dr$ and so, $|\nabla \widetilde u|=const$ on $\partial B_{\rho(t)}\cap\Dr$, for every $\rho(t)$. Moreover, since $\partial\Dr$ is not flat in $0$, we have the isoperimetric inequality
$$\HH^{1}\big(\partial B_\rho\cap\Dr\big)<\HH^{1}\big(\partial B_r\cap H\big),$$
for every $\rho$ and $r$ such that $|B_\rho\cap\Dr|=|B_r\cap H|$. Thus, taking $f(t)=|\{u>t\}|$ we repeat the argument from Example \ref{sssector} obtaining
\[\begin{split}
\lambda_1(B_r\cap H;H)=\int_{H}|\nabla u|^2\,dx
&=\int_0^{+\infty}\left(|f'(t)|^{-1}\,\HH^{1}\big(\{u=t\}\cap H\big)^2\right)\,dt\\
&>\int_0^{+\infty}\left(|f'(t)|^{-1}\,\HH^{1}\big(\{\widetilde u=t\}\cap\Dr\big)^2\right)\,dt\\
&=\int_{\Dr}|\nabla\widetilde u|^2\,dx\ge\lambda_1(B_\rho\cap \Dr;\Dr),
\end{split}\]
which concludes the existence part. The regularity of the free boundary of the optimal sets follows by the result from \cite{brla} and the orthogonality to $\partial\Dr$ can be obtained as in Remark \ref{lb1orth}.
\end{proof}

\ack This work is part of the project 2010A2TFX2 {\it``Calcolo delle Variazioni''} funded by the Italian Ministry of Research and University. The first author is member of the Gruppo Nazionale per l'Analisi Matematica, la Probabilit\`a e le loro Applicazioni (GNAMPA) of the Istituto Nazionale di Alta Matematica (INdAM).

%%%%%%%%%%%%%%%%%%%%%%%%%%%%%%%%%%%%%%%%%%%%%%%%%%

\end{document}